\documentclass{bmcart}

\usepackage{hyperref}
\usepackage{amsmath,amsthm,amsfonts,amssymb,amscd,mathtools}
\usepackage{booktabs}
\usepackage[T1]{fontenc}
\usepackage[utf8]{inputenc} 

\usepackage{cite}
\numberwithin{equation}{section}
\usepackage[left,mathlines]{lineno}
\usepackage{etoolbox} 
\newcommand*\linenomathpatchAMS[1]{%
  \expandafter\pretocmd\csname #1\endcsname {\linenomathAMS}{}{}%
  \expandafter\pretocmd\csname #1*\endcsname{\linenomathAMS}{}{}%
  \expandafter\apptocmd\csname end#1\endcsname {\endlinenomath}{}{}%
  \expandafter\apptocmd\csname end#1*\endcsname{\endlinenomath}{}{}%
}

\expandafter\ifx\linenomath\linenomathWithnumbers
  \let\linenomathAMS\linenomathWithnumbers
  \patchcmd\linenomathAMS{\advance\postdisplaypenalty\linenopenalty}{}{}{}
\else
  \let\linenomathAMS\linenomathNonumbers
\fi

\linenomathpatchAMS{equation} 
\linenomathpatchAMS{gather}
\linenomathpatchAMS{multline}
\linenomathpatchAMS{align}
\linenomathpatchAMS{alignat}
\linenomathpatchAMS{flalign}

\usepackage[output-decimal-marker={.}]{siunitx}

\def\re{\mathbb R}
\def\na{\mathbb N}

\def\lV{\left\Vert }
\def\rV{\right\Vert }
\def\lv{\left\vert }
\def\rv{\right\vert}

\DeclareMathOperator{\Id}{Id}

\DeclareMathOperator{\aff}{aff}
\DeclareMathOperator{\dist}{dist}

\DeclareMathOperator{\circum}{circ}

\usepackage[capitalize,nameinlink]{cleveref}
\usepackage{autonum}
\newtheorem{theorem}{Theorem}[section]
\newtheorem{lemma}[theorem]{Lemma}
\newtheorem{corollary}[theorem]{Corollary}
\newtheorem{proposition}[theorem]{Proposition}
\theoremstyle{definition}
\newtheorem{definition}[theorem]{Definition}

\newtheorem{example}[theorem]{Example}

\theoremstyle{remark}

\Crefname{lemma}{Lemma}{Lemmas}
\Crefname{proposition}{Proposition}{propositions}
\Crefname{proposition}{Proposition}{Propositions}
\Crefname{corollary}{Corollary}{Corollaries}

\usepackage[shortlabels]{enumitem}
\newlist{lista}{enumerate}{1}
\setlist[lista]{label=\alph*., nosep,leftmargin=*,align=right}

\newlist{listi}{enumerate}{1}
\setlist[listi]{label={\upshape(\roman*\upshape)},leftmargin=*,align=right, widest=iii,nosep, format=\bf}

\def\includegraphics{}

\startlocaldefs
\endlocaldefs

\begin{document}

\begin{frontmatter}

\begin{fmbox}

\title{Circumcentering approximate reflections for solving the convex feasibility problem}


\author[
  addressref={fgv},                   
  email={guilhermehmaraujo@gmail.com}   
]{\inits{G.A.}\fnm{G.} \snm{Araújo}}
\author[
  addressref={impa},
  email={reza.arefidamghani@impa.br}
]{\inits{R.A.}\fnm{R.} \snm{Arefidamghani}}
\author[
  addressref={fgv},                   
  email={rogerbehling@gmail.com}   
]{\inits{R.B.}\fnm{R.} \snm{Behling}}
\author[
  addressref={niu},
  email={yunierbello@niu.edu}
]{\inits{Y.B.-C.}\fnm{Y.} \snm{Bello-Cruz}}
\author[
  addressref={impa},
  email={iusp@impa.br}
]{\inits{A.I.}\fnm{A.} \snm{Iusem}}
\author[
  corref={ufsc},
  addressref={ufsc},
  email={l.r.santos@ufsc.br}
]{\inits{L.-R.S.}\fnm{L.-R.} \snm{Santos}}


\address[id=fgv]{
  \orgdiv{School of Applied Mathematics},             
  \orgname{Funda\c{c}\~ao Get\'ulio Vargas},          
  \city{Rio de Janeiro, RJ},                              
  \cny{BR}                                    
}

\address[id=impa]{
  \orgname{Instituto de Matemática Pura e Aplicada},          
  \city{Rio de Janeiro, RJ},                              
  \cny{BR}                                    
}

\address[id=niu]{
  \orgdiv{Department of Mathematical Sciences},             
  \orgname{Northern Illinois University},          
  \city{DeKalb, IL},                              
  \cny{US}                                    
}
\address[id=ufsc]{%
  \orgdiv{ Department of Mathematics},
  \orgname{Federal University of Santa Catarina},
  \city{Blumenau, SC},
  \cny{BR}
}



\end{fmbox}


\begin{abstractbox}

\begin{abstract} 
The circumcentered-reflection method (CRM) has been applied for solving convex feasibility problems. CRM iterates by computing a circumcenter upon a composition of reflections with respect to convex sets. Since reflections are based on exact projections, their computation might be costly. In this regard,  we introduce the circumcentered approximate-reflection method (CARM), whose reflections rely on outer-approximate projections. The appeal of CARM is that, in rather general situations, the approximate projections we employ are available under low computational cost. 
  We derive convergence of CARM and linear convergence under an error bound condition. We also present successful theoretical and numerical comparisons of CARM to the original CRM, to the classical method of alternating projections (MAP) and to a correspondent outer-approximate version of MAP, referred to as MAAP. Along with our results and numerical experiments, we present a couple of illustrative examples.

\end{abstract}




\end{abstractbox}
%

\end{frontmatter}



\section{Introduction}\label{s0}

We consider the convex feasibility problem (CFP), consisting of 
finding a point in the intersection of a finite number of closed convex sets.
We are going to employ Pierra's product space reformulation~\cite{Pierra:1984} in order to reduce CFP to seeking a point common to a closed convex set and an affine subspace. 

Projection-based methods are usually utilized for solving CFP. Widely known are the Method of Alternating Projections (MAP) \cite{Kaczmarz:1937,Bauschke:1996}, the Douglas-Rachford method (DRM) \cite{Douglas:1956,Lions:1979} and the Cimmino method \cite{Cimmino:1938, Bauschke:1996}. Recently, the Circumcentered-Reflection method (CRM) has been developed as a powerful new tool for solving CFP, outperforming MAP and DRM. It was introduced in \cite{Behling:2018a,Behling:2018} and further enhanced in~\cite{Arefidamghani:2021,Bauschke:2018a,Bauschke:2021c,Bauschke:2020c,Bauschke:2021,Bauschke:2021a,Behling:2020,Dizon:2019,Dizon:2020,Ouyang:2020}. In particular, CRM was shown in \cite{Behling:2020} to converge to a solution of CFP and it was proven in \cite{Arefidamghani:2021} that linear convergence is obtained in the presence of an error bound condition.

Computing exact projections onto general convex sets can be, context depending, too demanding in comparison to solving the given CFP itself. 
Bearing this in mind, we present in this paper a version of CRM employing \emph{outer-approximate projections}. These approximate projections still enjoy some of the properties of the exact ones, having the advantage of being potentially more tractable. For instance, they cover the subgradient projections of Fukushima~\cite{Fukushima:1983}.


Consider closed  convex sets $K_1, \dots , K_m\subset\re^n$ with nonempty intersection and the CFP of finding a point  
\(x\in \bigcap_{i=1}^mK_i.\)  In the eighties, Pierra noted that this problem  is  directly related to  the problem of finding a point $\mathbf{x} \in \mathbf{K}\cap \mathbf{D}$, where $\mathbf{K} := K_1 \times\dots\times K_m\subset
\re^{nm}$ and the diagonal space $\mathbf{D}=\{(x,\dots,x): x\in\re^n\}\subset\re^{nm}$. In fact,  $x\in  \bigcap_{i=1}^mK_i$, if, and only if  $\mathbf{ x}=(x,\ldots, x) \in \mathbf{K} \cap \mathbf{D} $. Thus, if we solve any intersection problem featuring a closed convex set and an affine subspace, we cover the general CFP. Let us proceed in this direction by considering a closed convex set $K\subset \re^n$ and an affine subspace $U\subset \re^n$ with nonempty intersection. 
From now on, the CFP we are going to focus on is the one of tracking  a point in $K\cap U$.

We consider now two operators $A$ and $B:\re^n\to\re^n$ and define $T=A\circ B$.
Under adequate assumptions, the sequence $\{x^k\}_{k\in\na}\subset\re^n$ defined by 
\begin{equation}\label{e0}
x^{k+1}=T(x^k)=A(B(x^k))
\end{equation}
is expected
to converge to a common fixed point of $A$ and $B$. If the operators $A$ and $B$ are the projectors onto the convex sets $U$ and $K$, that is, $A=P_{U}, B=P_{K}$, then problem \eqref{e0} provides the iteration of the famous method of alternating projections (MAP). Moreover,  
the set of common fixed points of $A$ and $B$ in this case is precisely $K\cap U$ and MAP converges to a point in $K\cap U$ for any starting point in $\re^n$; see, for instance, \cite{Bauschke:1996}.

 The circumcentered-reflection method (CRM) introduced in \cite{Behling:2018a,Behling:2018} can be seen as an acceleration technique for the sequence defined by \eqref{e0}. We showed in \cite{Arefidamghani:2021} that indeed CRM achieves a better linear rate than MAP in the presence of an error bound. Moreover, in general,  there is abundant numerical evidence
that CRM outperforms MAP (see \cite{Behling:2018a,Behling:2020,Dizon:2019}).

Define the reflection operators $A^R, B^R:\re^n\to\re^n$ as $A^R=2A-\Id, B^R=2B-\Id$,
where $\Id$ stands for the identity operator in $\re^n$. The CRM operator $C:\re^n\to\re^n$ is defined
as  
\begin{equation}\label{ccoperator}C(x)= \circum(x, B^R(x), A^R(B^R(x))),\end{equation}
\emph{i.e.}, the circumcenter of the three points $x, B^R(x),A^R(B^R(x))$. The CRM sequence $\{x^k\}_{k\in\na}\subset\re^n$,
starting at some $x^0\in\re^n$, is then defined as
$x^{k+1}=C(x^k)$. For three non-collinear points $x,y,z\in\re^n$, the {\it circumcenter} $\circum(x,y,z)$ 
is the center of the unique two-dimensional circle 
passing through $x,y,z$ (or, equivalently, the point in the affine manifold $\aff\{x,y,z\}$ generated by $x,y,z$ and equidistant to these three points). 
In particular, if, $A=P_{U}$ and $B=P_K$, that is, $A^R=2P_{U}-\Id, B^R=2P_{K}-\Id$, the CRM sequence, 
starting at some $x^0\in U$, 
\begin{equation}\label{e01}
x^{k+1}=C(x^k)= \circum(x^k, B^R(x^k), A^R(B^R(x^k))),
\end{equation}
converges to a point in $K\cap U$ as long as the initial point lies in $U$. If in addition a certain error bound between $K$ and $U$
holds, then CRM converges linearly, and with a better rate than MAP. 



In this paper, we introduce approximate versions of MAP and CRM for solving CFP, which we call MAAP and CARM. The MAAP and CARM iterations are computed by \eqref{e0} and \eqref{e01} with $A$ being the exact projector onto $U$ and $B$ an approximate projector onto $K$. The approximation 
consists of replacing  at each iteration
the set $K$ by a larger set separating the current iterate from $K$. This separating scheme is rather general, and for a large family
of convex sets, includes
particular instances where the separating set is a half-space, or a Cartesian product of half-spaces, in which cases all the involved 
projections have a very low computational cost. One could fear that this significant reduction in the
computational cost per iteration could be nullified by a substantial slowing down of the process as a whole, through a 
deterioration of the convergence speed. However, we show that this is not necessarily the case. Indeed, we prove that under error bound conditions
separating schemes are available so that MAAP and CARM enjoy linear convergence rates, with the linear rate of CARM being strictly better than MAAP. Our numerical experiments confirm these statements, and more than that, they show CARM outperforming MAP, CRM and MAAP in terms of computational time.


\section{Preliminaries}\label{s1}
 
We recall first the definition of Q-linear and R-linear convergence.

\begin{definition}\label{d1}
Consider a sequence $\{z^k\}_{k\in\na}\subset\re^n$ converging to $z^*\in\re^n$. Assume that $z^k\ne z^*$ for all $k\in\na$.
Let $q\coloneqq \limsup_{k\to\infty}\frac{\lV z^{k+1}-z^*\rV}{\lV z^k-z^*\rV}$. Then, the 
sequence $\{z^k\}_{k\in\na}$ converges
\begin{listi}
      \item {Q}-superlinearly, if $q=0,$
   
   \item {Q}-linearly, if $q\in(0,1),$
     
   \item {Q}-sublinearly, if $q\ge 1.$

\end{listi}
Let
$r\coloneqq \limsup_{k\to\infty}\lV z^k-z^*\rV^{1/k}$. Then, the 
sequence $\{z^k\}_{k\in\na}$ converges
\begin{listi}[resume]
\item {R}-superlinearly, if $r=0,$

\item {R}-linearly, if $r\in(0,1),$

\item  {R}-sublinearly, if  $r\ge 1.$

\end{listi}
 The values
$q$ and $r$ are called asymptotic constants of $\{z^k\}_{k\in\na}$.
\end{definition} 

It is well known that Q-linear convergence implies R-linear convergence (with the same asymptotic constant), but the
converse statement does not hold true \cite{Ortega:2000}.

We remind now the notion of Fej\'er monotonicity.

\begin{definition} 
A sequence $\{z^k\}_{k\in\na}\subset\re^n$ is Fej\'er monotone with respect to a nonempty closed convex set $M\subset \re^n$ when  
$\lV z^{k+1}-y\rV\le\lV z^k-y\rV$ for all $k\in\na$ and $y\in M$.
\end{definition}

\begin{proposition}\label{p1}
Suppose that the sequence $\{z^k\}_{k\in\na}$ is Fej\'er monotone with respect to the closed convex set  $M\subset \re^n$.  Then, 
\begin{listi}
\item
$\{z^k\}_{k\in\na}$ is bounded.
\item if a cluster point $\bar z$ of $\{z^k\}_{k\in\na}$ belongs to $M$ we have  $\lim_{k\to\infty}z^k=\bar z$.
\item if $\{z^k\}_{k\in\na}$ converges to $\bar z$ we get $\lV z^k-\bar z\rV\le 2\dist(z^k,M)$.
\end{listi}
\end{proposition} 

\begin{proof} 
See Theorem 2.16 in \cite{Bauschke:1996}. 
\end{proof}

Next we introduce the separating operator needed for the approximate versions of MAP and CRM, namely MAAP and CARM.

\begin{definition}\label{dd1}
Given a closed and convex set $K\subset\re^n$, a {\it separating operator} for $K$ is a point-to-set mapping 
$S:\re^n\to{\cal P}(\re^n)$ satisfying:
\begin{listi}
\item[A1)] $S(x)$ is closed and convex for all $x\in\re^n$.
\item[A2)] $K\subset S(x)$ for all $x\in\re^n$.
\item[A3)] If a sequence $\{z^k\}_{k\in\na}\subset\re^n$ converges to $z^*\in\re^n$ and 
$\displaystyle \lim_{k\to\infty}\dist(z^k,S(z^k))=0$ then $z^*\in K$.
\end{listi}
\end{definition}

 We have the following immediate result regarding \Cref{dd1}.

\begin{proposition}\label{pq1} If $S$ is a separating operator for $K$ then
$x\in S(x)$ if and only if $x\in K$.
\end{proposition}

\begin{proof}
The ``if''  statement follows from A2. For the ``only if'' statement, take $x\in S(x)$, consider the constant
sequence $z^k=x$ for all $k\in\na$, which converges to $x$, and apply A3.
\end{proof}   

\Cref{pq1} implies that if $x\notin K$ then $x\notin S(x)$, which, in view of A2, indicates that
the set $S(x)$ separates indeed $x$ from $K$. The separating sets $S(x)$ will provide to the approximate projections that we are going to employ throughout the paper.

Several notions of separating operators have been introduced in the literature; see, \emph{e.g.}, \cite[Section 2.1.13]{Cegielski:2012} and references therein. Our definition is a point-to-set version of the separating operators in~\cite[Definition 2.1]{Cegielski:2010}. It encompasses not only hyperplane-based separators as the ones in the seminal work by Fukushima~\cite{Fukushima:1983}, considered next in \Cref{ex1}, but also more general situations. Indeed, in \Cref{ex2}, $S(x)$ is the Cartesian product of half-spaces, which is not a half-space.

For the family of convex sets in \Cref{ex1,ex2},  we get  both explicit separating operators complying with \cref{dd1} and closed formulas for projections onto them.

\begin{example}\label{ex1} Assume that $K=\{x\in\re^n: g(x)\le 0\}$, where $g:\re^n\to\re$ is convex. Define
\begin{equation}\label{aa3}
S(x)=\begin{cases} K, & {\rm if}\,\,x\in K\\
\{z\in\re^n:u^{\top}(z-x)+g(x)\leq 0\},& {\rm otherwise},
\end{cases} 
\end{equation}
where $u\in \partial g(x)$ is an arbitrary subgradient of $g$ at $x$.
\end{example}

We mention that any closed and convex set $K$
can be written as the $0$-sublevel set of a convex, and even smooth function $g$, for instance, $g(x)=\dist(x,K)^2$,
but in general this is not advantageous, because for this $g$ it holds that $\nabla g(x)=2(x-P_K(x))$, so
that $P_K(x)$, the exact projection of $x$ onto $K$,  is needed for computing the separating half-space, and nothing
has been won. The scheme is interesting when the function $g$ has easily computable gradient or subgradients.
For instance, in the quite frequent case in which $K=\{x\in\re^n: g_i(x)\le 0\,\,\,(1\le i\le \ell)\}$, where the $g_i$'s 
are convex and smooth, we can take $g(x)=\max_{1\le i\le\ell}g_i(x)$, and the subgradients of $g$ are easily obtained 
from the gradients of the $g_i$'s.

\begin{example}\label{ex2} Assume that $\mathbf{K}=K_1\times\dots\times K_m\subset \re^{nm}$, where $K_i\subset\re^n$
is of the form $K_i=
\{x\in\re^n: g_i(x)\le 0\}$ and $g_i:\re^n\to\re$ is convex for $1\le i\le m$. Write $x\in\re^{nm}$ as
$x=(x^1,\dots ,x^m)$ with $x^i\in\re^n (1\le i\le m)$.  We define the separating operator 
$\mathbf{S}:\re^{nm}\to{\cal P}(\re^{nm})$ as $\mathbf{S}(x)=S_1(x^1)\times\dots\times S_m(x^m)$, with
\begin{equation}\label{aa4}
S_i(x^i)=\begin{cases} K_i, &{\rm if}\,\,x^i\in K_i,\\
\{z\in\re^n:(u^i)^{\top}(z-x^i)+g_i(x^i)\leq 0\}, & {\rm otherwise},
\end{cases} 
\end{equation}
where $u^i\in \partial g_i(x^i)$ is an arbitrary subgradient of $g_i$ at $x^i$.
\end{example}
\Cref{ex2} is suited for the reduction of Simultaneous projection method (SiPM) for $m$ convex sets in $\re^n$ to MAP regarding two convex sets in $\re^{nm}$.  Note that in \Cref{ex1}, $S(x)$ is either $K$ or a half-space, and the same holds for the
sets $S_i(x^i)$ in \Cref{ex2}. We prove next that the separating operators $S$ and $\mathbf{S}$
defined in \Cref{ex1,ex2} satisfy assumptions A1--A3.

\begin{proposition} \label{pq2}
The separating operators $S$ and $\mathbf{S}$
defined in \Cref{ex1,ex2} satisfy assumptions {\rm A1--A3}.
\end{proposition}

\begin{proof}
We start with $S$ as in \Cref{ex1}. First we observe that if $x\notin K$ then
all subgradient of $g$ at $x$ are nonzero: since $K$ is assumed nonempty, there exists points where $g$ is nonpositive,
so that $x$, which satisfies $g(x) >0$, cannot be a minimizer of $g$, and hence $0\notin\partial g(x)$, \emph{i.e.}, $u\neq 0$ for all 
$u\in\partial g(x)$. Regarding A1, $S(x)$ is either equal to $K$ or to a half-space, both of which are closed and convex. 

For
A2, obviously it holds for $x\in K$. If $x\notin K$, we take $z\in K$, and conclude, taking into account the fact that $z\in K$ and 
the subgradient inequality, that
$u^{\top}(x-z)+g(x)\le g(z)\le 0$, implying that $z\in S(x)$, in view of \eqref{aa3}. 

We deal now with A3. Take a sequence $\{z^k\}_{k\in\na}$
converging to some $z^*$ such that $\lim_{k\to\infty}\dist(z^k,S(z^k))$ $=0$. We must prove that $z^*\in K$.
If some subsequence of $\{z^k\}_{k\in\na}$ is contained in $K$ then
$z^*\in K$, because $K$ is closed. Otherwise, for large enough $k$, $S(z^k)$ is a half-space. 
It is well known, and easy to check, that the projection $P_H$ onto a half-space 
$H=\{y\in\re^n: a^{\top}y\le\alpha\}\subset\re^n$, with $a\in\re^n, \alpha\in\re$, is given by
\begin{equation}\label{ee1}
P_H(x)= x-\lV a\rV^{-2}\max\{0, \alpha-a^{\top}x\}a.
\end{equation}
Denote by $P^{S_k}$ the projection onto $S(z^k)$. By \eqref{ee1},
$P^{S_k}(z)=z- \lV u^k\rV^{-2}\max\{0, g(z)\}u^k$,  so that
\[\dist(z^k,S(z^k))=\lV z^k-P^{S_k}(z^k)\rV=\lV u^k\rV^{-1}\max\{0,g(z^k)\}.\] Note that $\{z^k\}_{k\in\na}$ is bounded, because it is convergent.
Since the subdifferential operator $\partial g$ is locally bounded in the interior of the domain of $g$, which here we take
as $\re^n$, there exists $\mu >0$ so that $\lV u^k\rV\le\mu$ for all $k$ and all $u^k\in\partial g(z^k)$. Hence,
$\dist(z^k,S(z^k))\ge\mu^{-1}\max\{0,g(z^k)\}\ge 0$. Since by assumption $\lim_{k\to\infty}\dist(z^k,S(z^k))=0$,  and $g$, being convex,
is continuous, we get that
$0=\lim_{k\to\infty}\mu^{-1}\max\{0,g(z^k)\}=\mu^{-1}\max\{0,g(z^*)\}$, implying that $0=\max\{0,g(z^*)\}$,\emph{i.e.}, $g(z^*)\le 0$,
so that $z^*\in K$ and A3 holds.

Now we consider $\mathbf{S}$ as in \Cref{ex2}. As before, if $x^i\notin K_i$ then $S_i(x^i)$ is indeed a half-space in $\re^n$.
Concerning A1--A3, 
A1 holds because $\mathbf{S}(x)$ is the Cartesian product of closed and convex sets
(either $K_i$ or a half-space in $\re^n$). For A2, take $(x^1, \dots, x^m)\in\mathbf{K}$. If $x^i\in K_i$, then $x^i\in S_i(z^i)=K_i$.
Otherwise, we take $z^i\in K_i$, and invoking again the subgradient inequality we get $(u^i)^{\top}(x^i-z^i)+g(x^i)\le g(z^i)\le 0$ 
implying that $z^i\in S_i(x^i)$, \emph{i.e.} $K_i\subset S_i(X^i)$ for all $i$, and the result follows taking into account the definitions
of $\mathbf{K}$ and $\mathbf{S}$. For A3, note that $\lim_{k\to\infty}\dist(z^k,\mathbf{S}(z^k))=0$ if and only if 
$\lim_{k\to\infty}\dist(z^{k,i},S_i(z^{k,i}))=0$ for $1\le i\le m$, where $z^k=(z^{k,1},\dots ,z^{k,m})$ with $z^{k,i}\in\re^n$.
Then, the result follows with the same argument as in \Cref{ex1}, with $z^{k, i}, S_i, g_i$ substituting for $z^k, S, g$.
\end{proof} 

\section{Convergence results for MAAP and CARM}\label{s2}       

We recall now the definitions of MAP and CRM, and introduce the formal definitions of MAAP and CARM.
Consider a closed convex set $K\subset\re^n$ and an affine manifold $U\subset\re^n$. We remind that an affine manifold is a set of the form $\{x\in\re^n: Qx=b\}$, for some $Q\in \re^{n\times n}$ and some $b\in \re^n$. 

Let $P_K, P_U$ be the projections onto $K,U$ respectively and define $R_K,R_U, T, C:\re^n\to\re^n$ as
\begin{equation}\label{e1}
\begin{gathered}
R_K=2P_K-\Id, \quad R_U=2P_U-\Id, \\
 T = P_U\circ P_K, \quad C(x)=\circum(x,R_K(x), R_U(R_K(x))),
\end{gathered}
\end{equation} 
where $\Id$ is the identity operator in $\re^n$
and $\circum(x,y,z)$ is the circumcenter of $x,y,z$, {\emph{i.e.}}, the point in the affine hull of $x,y,z$ equidistant to them.
We remark that $\circum(x,y,y)= (1/2)(x+y)$ (in this case the affine hull is the line through $x,y$), and $\circum(x,x,x)=x$
(the affine hull being the singleton $\{x\}$).

Then, starting from any $x^0\in\re^n$, MAP generates a sequence $\{x^k\}_{k\in\na}\subset\re^n$ according to 
\begin{equation}\label{aa5}
x^{k+1}=T(x^k),
\end{equation}
and, starting with $x^0\in U$, CRM generates a sequence $\{x^k\}_{k\in\na}\subset\re^n$ given by   
\begin{equation}\label{aa6}
x^{k+1}=C(x^k).
\end{equation}
 
For MAAP and CARM, we assume that $S:\re^n\to{\cal P}(\re^n)$ is a separating operator for $K$ satisfying A1--A3,
we take $P_U$ as before and define $P^S$ as the operator given by $P^S(x):=P_{S(x)}(x)$, where $P_{S(x)}$ is the 
projection onto $S(x)$.  

Take $R_U$ as in \eqref{e1}, and define $R^S, T^S, C^S:\re^n\to\re^n$ as
\begin{equation}\label{aa7}
\begin{gathered}
 T^S = P_U\circ P^S, \quad R^S=2P^S-\Id, \quad C^S(x)=\circum\left(x,R^S(x), R_U(R^S(x))\right).
\end{gathered}
\end{equation} 

Then, starting from any $x^0\in\re^n$, MAAP generates a sequence $\{x^k\}_{k\in\na}\subset\re^n$ according to 
\begin{equation}\label{aa8}
x^{k+1}=T^S(x^k),
\end{equation} 
and, starting with $x^0\in U$, CARM generates a sequence $\{x^k\}_{k\in\na}\subset\re^n$ given by   
\begin{equation}\label{aa9}
x^{k+1}=C^S(x^k).
\end{equation} 

We observe now that the ``trivial'' separating operator $S(x)=K$ for all $x\in\re^n$ satisfies A1-A3,
and that in this case we have $T^S=T, C^S=C$, so that MAP, CRM are particular instances of MAAP, CARM
respectively. Hence, the convergence analysis of the approximate algorithms encompasses the exact ones.
Global convergence of MAP is well known (see, \emph{e.g.} \cite{Cheney:1959}) and the corresponding result for CRM has
been established in \cite{Behling:2020}. The following propositions follow quite closely the corresponding
results for the exact algorithms, the difference consisting in the replacement of the set $K$ by the
separating set $S(x)$. However, some care is needed, because $K$ is fixed, while $S(x)$ changes along the algorithm,
so that we present the complete analysis for the approximate algorithms MAAP and CARM.

\begin{proposition}\label{p2}
For all $z\in K\cap U$ and all $x\in\re^n$, it holds that
\begin{equation}\label{aa10}
\lV T^S(x)-z\rV ^2\le \lV z-x\rV^2-\lV T^S(x)-P^S(x)\rV^2-\lV P^S(x)-x\rV^2,
\end{equation} 
with $T^S$ as in \eqref{aa7}. 
\end{proposition}

\begin{proof}
The projection operator $P_M$ onto any closed and convex set $M$ is known to be firmly nonexpansive \cite[Proposition 4.16]{Bauschke:2017}, that is, 
\begin{equation}\label{aa11}
\lV P_M(x)-y\rV^2\le\lV x-y\rV^2-\lV P_M(x)-x\rV^2
\end{equation} 
for all $x\in\re^n$ and all $y\in M$.

Applying consecutively \eqref{aa11} with $M=U$ and $M=S(x)$, and noting that for $z\in K\cap U$, we get $z\in U$ and also $z\in K\subset S(x)$
(due to Assumption A2), we obtain \eqref{aa10}.
\end{proof}

A similar result for operator $C^S$ is more delicate due to the presence of the reflections and the circumcenter and requires some intermediate results. We follow closely the analysis for operator $C$ presented in \cite{Behling:2020}.

The crux of the convergence analysis of CRM, performed in \cite{Behling:2020}, is the remarkable observation  that for $x\in U\setminus K$, $C(x)$ is indeed the projection of $x$ onto a half-space $H(x)$ separating $x$ from $K\cap U$. Next, we extend this result to $C^S$.

\begin{proposition}\label{p4}
Let $U,H\subset\re^n$ be an affine manifold and a subspace respectively, such that $H\cap U\ne\emptyset$. 
Denote as $P_H, R_H:\re^n\to\re^n$ the projection and the reflection with respect to $H$, respectively. Then, 
\begin{listi}
\item$P_{H\cap U}(x)=\circum(x,R_U(x), R_H(R_U(x))$ for all $x\in U$,
\item $\circum(x,R_U(x), R_H(R_U(x))\in U$ for all $x\in U$.
\end{listi}
\end{proposition}
\begin{proof}
See Lemma 3 in \cite{Behling:2020}.
\end{proof}
 
\Cref{p4} means that when the sets in CFP are an affine manifold and a hyperplane, 
CRM indeed converges in one step, which is a first indication of its superiority over MAP, which certainly 
does not enjoy this one-step
convergence property, but also points to the main weakness of CRM, namely that for its convergence we may
replace $H$ by a general closed and convex set, but the other set must be kept as an affine manifold.
 
\begin{lemma}\label{l1} 
Define $H(x)\subset\re^n$ as
\begin{equation}\label{aa13}
H(x):=\begin{cases} K,& {\rm if}\,\,x\in K\\
\left\{z\in\re^n:(z-P^S(x))^{\top}(x-P^S(x))\le 0\right\},& {\rm otherwise}.
\end{cases}
\end{equation}
Then, for all $x\in U, C^S(x)=P_{H(x)\cap U}(x)$.
\end{lemma}

\begin{proof}
Take $x\in U$. If $x\in K$, then $x\in S(x)$ by A2, and it follows that $R_U(x)=R^S(x)=x$, so that
$C^S(x)=$ $\circum(x,x,x)=x$. Also, $P_{H(x)}(x)=P_K(x)=x$ by \eqref{aa13},
and the result holds. Assume that $x\in U\setminus K$, so that $H(x)$ is the half-space in \eqref{aa13}. 

In view of \eqref{aa13}, we get, using \eqref{ee1} with $a=x-P^S(x), \alpha=(x-P^S(x))^{\top}P^S(x)$, that
$P_{H(x)}(x)=P^S(x)$. It follows from the definition of the reflection operator $R^S$ that
\begin{equation}\label{aa14}
R^S(x)=R_{H(x)}(x).
\end{equation} 
Now, by \eqref{aa7} and \eqref{aa14}, 
$$C^S(x)=\circum(x,R^S(x), R_U(R^S(x)))=\circum(x, R_{H(x)}, R_U(R_{H(x)}(x))).$$ 
Since $U$ is an affine manifold and $H(x)$ is a half-space, we can apply \Cref{p4} and conclude that
$C^S(x)=P_{H(x)\cap U}(x)$, proving the last statement of the lemma. 
By assumption, $x\in U$, so that $P_{H(x)\cap U}(x)=P_{H(x)}(x)$, establishing the result.
\end{proof}

This rewriting of the operator $C^S$ as a projection onto a half-space (which varies with the argument of $C^S$),
allows us to obtain the result for CARM analogous to \Cref{p2}. 
 
\begin{proposition}\label{p5}
For all $z\in K\cap U$ and all $x\in U$, it holds that
\begin{listi}
\item
$\label{aa15}
\lV C^S(x)-z\rV ^2\le \lV z-x\rV^2-\lV C^S(x)-x\rV^2,
$
with $C^S$ as in \eqref{aa7}.
\item $C^S(x)\in U$ for all $x\in U$. 
\end{listi}
\end{proposition}

\begin{proof}
For (i), take $z\in K\cap U$ and $x\in U$. By \Cref{l1}, $C^S(x)=P_{H(x)}(x)$ for all $x\in U$. 
Since $z\in K\subset H(x)$, we can apply \eqref{aa11} with $M=H(x)$, obtaining
$\lV P_{H(x)}(x)-z\rV^2\le\lV x-z\rV^2-\lV P_{H(x)}(x)-x\rV^2$, which gives the result, 
invoking again \Cref{l1}.  Item (ii) follows from \Cref{p4} and \Cref{l1}. 
\end{proof}

\Cref{p2,p5} allow us to prove convergence of the MAAP and CARM sequences respectively,
using the well known Fejér monotonicity argument.

\begin{theorem}\label{t1}
Consider a closed and convex set $K\subset\re^n$ and an affine manifold $U\subset\re^n$ such that $K\cap U\neq \emptyset.$ Consider also a separating operator $S$
for $K$ satisfying Assumptions A1--A3. 
Then the sequences generated by either MAAP or CARM, starting from any initial point in the MAAP case, and from
a point in $U$ in the CARM case, are well defined, contained in $U$, Fejér monotone with respect to $K\cap U$, 
convergent, and 
their limits belong to $K\cap U$, {\emph{i.e.}}, they solve
CFP. 
\end{theorem}

\begin{proof}
Let first $\{x^k\}_{k\in\na}$ be the sequence generated by MAAP, {\emph{i.e.}} $x^{k+1}=T^S(x^k)$. Take any $z\in K\cap U$.
Then, by \Cref{p2},
\begin{align}
\lV x^{k+1}-z\rV^2 & \le\lV x^k-z\rV^2-\lV P_U(P^S(x^k))-P^S(x^k)\rV^2-\lV P^S(x^k)-x^k\rV^2 \\ &\le \lV x^k-z\rV^2,\label{aa16}
\end{align}
and so $\{x^k\}_{k\in\na}$ is Fejér monotone with respect to $K \cap U$. By the Definition of $T^S$ in \eqref{aa7},
$\{x^k\}_{k\in\na}\subset U$. By \Cref{p1}(i), $\{x^k\}_{k\in\na}$ is bounded.
Also, $\{\lV x^k-z\rV\}_{k\in\na}$ is nonincreasing and nonnegative, therefore convergent, and thus the difference between
consecutive iterates converges to 0. Hence, rewriting \eqref{aa16} as
\[
\lV P_U(P^S(x^k))-P^S(x^k)\rV^2+\lV P^S(x^k)-x^k\rV^2\le\lV x^k-z\rV^2-\lV x^{k+1}-z\rV^2,
\]
we conclude that 
\begin{equation}\label{aa17}
\lim_{k\to\infty}\lV P_U(P^S(x^k))-P^S(x^k)\rV^2=0,
\end{equation}
and
\begin{equation}\label{aa18}
\lim_{k\to\infty}\lV P^S(x^k)-x^k\rV^2=0.
\end{equation}
Let  $\bar x$ be a cluster point of $\{x^k\}_{k\in\na}$ and $\{x^{j_k}\}_{j_k\in\na}$ a subsequence of $\{x^k\}_{k\in\na}$ convergent to $\bar x$.
By \eqref{aa18}, $\lim_{k\to\infty}\dist(x^{j_k},S(x^{j_k}))=0$. By Assumption A3 on the separating operator $S$,
$\bar x\in K$. It follows also from \eqref{aa18} that $\lim_{k\to\infty}P^S(x^{j_k})=\bar x$. By \eqref{aa17} and
continuity of $P_U$, $P_U(\bar x)=\bar x$, so that $\bar x\in U$ and therefore $\bar x\in K\cap U$. By \Cref{p1}(ii), $\bar x=\lim_{k\to\infty}x^k$, completing the proof for the case of MAAP.

Let now $\{ x^k\}_{k\in\na}$ be the sequence generated by CARM with $x^0\in U$. By \Cref{l1}, whenever $x^k\in U$,
$x^{k+1}$ is the projection onto a closed and convex set, namely $H(x^k)$, and hence it is well defined.
Since $x^0\in U$ by assumption, the whole sequence is well defined, and using recursively \Cref{p5}(ii),
we have that $\{x^k\}_{k\in\na}\subset U$. Now we use \Cref{p4}, obtaining, for any
$z\in K\cap U$,    
$$
\lV x^{k+1}-z\rV^2\le\lV x^k-z\rV^2-\lV C^S(x^k)-x^k\rV^2\le \lV x^k-z\rV^2,
$$
so that again $\{ x^k\}_{k\in\na}$ is Fejér monotone with respect $K\cap U$, and henceforth bounded. Also, with the same argument as before,
we get 
\begin{equation}\label{aa19}
\lim_{k\to\infty}\lV x^{k+1}-x^k\rV=\lim_{k\to\infty} \lV C^S(x^k)-x^k\rV=0. 
\end{equation}
In view of \eqref{aa19} and the definition of circumcenter, $\lV x^{k+1}-x^k\rV=\lV x^{k+1}-R^S(x^k)\rV$, so
that $\lim_{k\to\infty}\lV x^{k+1} -R^S(x^k)\rV =0$ implying that $\lim_{k\to\infty}\lV x^{k+1}-P^S(x^k)\rV =0$.
Thus, since $\lV x^k-P^S(x^k)\rV\le \lV x^k-x^{k+1}\rV+\lV x^{k+1}-P^S(x^k)\rV$, we get that
\begin{equation} \label{aa20}
0=\lim_{k\to\infty}\lV x^k-P^S(x^k)\rV =\lim_{k\to\infty}\dist(x^k,S(x^k)).
\end{equation}
Let $\bar x$ be any cluster point of $\{x^k\}_{k\in\na}$. Looking at \eqref{aa20} along a subsequence of $\{x^k\}_{k\in\na}$ converging to $\bar x$,
and invoking Assumption A3 of the separating operator $S$, we conclude
that $\bar x\in K$. Since $\{x^k\}_{k\in\na}\subset U$, we get that all cluster points of
$\{x^k\}_{k\in\na}$ belong to $K\cap U$, and then, using \Cref{p1}(ii), we get that $\lim_{k\to\infty}x^k=\bar x\in
K\cap U$, establishing the convergence result for CARM.
\end{proof}

\section{Linear convergence rate of MAAP and CARM under a local error bound assumption}\label{s3}

In \cite{Arefidamghani:2021}, the following  {\it global error bound} assumption on the sets $K,U$, denoted as \eqref{eb}, was considered:

\begin{enumerate}[label=EB),font=\bfseries,leftmargin=*,ref=EB]
\item \label[assumption]{eb}  There exists $\bar\omega >0$ such that $\dist(x,K)\ge\bar\omega \dist(K\cap U)$ for all $x\in U$.
\end{enumerate}

Let us comment on the connection between
\eqref{eb} and other notions of error bounds which have been introduced in the past, all of them related to regularity assumptions imposed on the solutions of certain problems. If the problem at hand consists of solving $H(x)=0$ with a smooth $H:\re^n\to\re^m$, a classical regularity condition demands that $m=n$ and the Jacobian matrix of $H$ be nonsingular at a solution $x^*$, in which case, Newton's method, for instance, is known to enjoy superlinear or quadratic convergence. This condition implies local uniqueness of the solution $x^*$. For problems
with nonisolated solutions, a less demanding assumption is the
notion of {\it calmness} (see \cite{Rockafellar:2004}, Chapter 8, Section F), which requires that 
\begin{equation}\label{fufu}
\frac{\lV H(x)\rV}{\dist(x,S^*)}\ge\omega
\end{equation}
for all $x\in\re^n\setminus S^*$ and some $\omega>0$, where $S^*$ is the solution set, \emph{i.e.}, the set of zeros of $H$.
 Calmness, also called upper-Lipschitz continuity (see \cite{Robinson:1982}), is a classical example of error bound, and it holds in many situations (\emph{e.g.}, when $H$ is affine, by virtue of Hoffman's Lemma, \cite{Hoffman:1952}). It implies that the solution set is locally a Riemannian manifold (see \cite{Behling:2013b}), and it has been used for establishing superlinear convergence of Levenberg-Marquardt methods in \cite{Kanzow:2004}.

When dealing with convex feasibility problems, as in this paper, it seems reasonable to replace the numerator of \eqref{fufu} by the distance from $x$ to some of the convex sets, giving rise to \eqref{eb}.
Similar error bounds for feasibility problems can be found, for instance, in \cite{Bauschke:1993,Bauschke:1996a,Drusvyatskiy:2015,Kruger:2018}.

Under \eqref{eb}, it was proved in \cite{Arefidamghani:2021} that MAP converges linearly, with asymptotic constant bounded above
by $\sqrt{1-\bar\omega^2}$, and that CRM also converges linearly, with a better upper bound for the asymptotic constant,
namely $\sqrt{(1-\bar\omega^2)/(1+\bar\omega^2)}$. In this section, we will prove similar results for MAAP and CARM,
assuming a {\it local error bound} related not just to $K,U$, but also to the separating operator $S$. The local error bound,
denoted as \eqref{leb} is defined as:

\begin{enumerate}[label=LEB),font=\bfseries,leftmargin=*,ref=LEB] 
\item \label[assumption]{leb} There exists a set $V\subset\re^n$ and a scalar $\omega >0$ such that 
$$\dist(x,S(x))\ge\omega \dist(x,K\cap U)\quad \mbox{for all} \quad x\in U\cap V.$$
\end{enumerate}

We reckon that \eqref{leb} becomes meaningful, and relevant for establishing convergence rate results, only when the set $V$ contains the tail of the sequence generated by the algorithm; otherwise it might be void (\emph{e.g.,} it holds trivially, with any $\omega$, when $U\cap V=\emptyset$). In order to facilitate the presentation, we opted for 
introducing additional conditions on $V$ in our convergence results,
rather than in the definition of \eqref{leb}. 

The use of a local error bound instead of a global one makes sense, because the definition of linear convergence rate deals
only with iterates $x^k$ of the generated sequence with large enough $k$, and the convergence of the sequences of interest
has already been established in \Cref{t1}, so that only points close enough to the limit $x^*$ of the sequence matter. 
In fact, the convergence rate analysis for MAP and CRM in \cite{Arefidamghani:2021} holds, without any substantial change, under a local, 
rather than global, error bound.
 
The set $V$ could be expected to be a neighborhood of the limit $x^*$ of the sequence, but we do not specify it for the time 
being, because for the prototypical example of separating operator, {\emph{i.e.}}, the one in \Cref{ex1} of \Cref{s2}, it will have, 
as we will show later, a slightly more complicated structure: a ball centered at $x^*$ minus a certain ``slice''.  

We start with the convergence rate analysis for MAAP.

\begin{proposition}\label{p6}
Assume that $K,U$ and the separating operator $S$ satisfy \eqref{leb}. Consider $T^S:\re^n\to\re^n$ as in \eqref{e1}. 
Then, for all $x\in U\cap V$,
\begin{equation}\label{aa21}
(1-\omega^2)\lV x-P_{K\cap U}(x)\rV^2\ge\lV T^S(x)-P_{K\cap U}(T^S(x))\rV^2,
\end{equation}
with $\omega$ as in Assumption \eqref{leb}.
\end{proposition}

\begin{proof}
By \Cref{p2}, for all $z\in K\cap U$ and all $x\in\re^n$,
\begin{align}
\lV T^S(x)-z\rV ^2&\le \lV z-x\rV^2-\lV T^S(x)-P^S(x)\rV^2-\lV P^S(x)-x\rV^2\\&\le\lV x-z\rV^2-\lV P^S(x)-x\rV^2.\label{aa22}
\end{align}
Note that $\lV P^S(x)-x\rV=\dist(x,S(x))$ and that $\lV T^S(x)-P_{K\cap U}(T^S(x))\rV\le\lV T^S(x)-z\rV$ by definition of $P_{K\cap U}$.
Take $z= P_{K\cap U}(x)$, and get from \eqref{aa22}
\begin{align}
\lV T^S(x)-P_{K\cap U}(T^S(x))\rV^2&\le\lV T^S(x)-P_{K\cap U}(x)\rV^2\\&\le\lV x-P_{K\cap U}(x)\rV^2-\dist(x,S(x))^2.\label{aa24}
\end{align}
Take now $x\in U\cap V$ and invoke \eqref{leb} to get from \eqref{aa24}
\begin{align}
\lV T^S(x)-P_{K\cap U}(T^S(x))\rV^2&\le\lV x-P_{K\cap U}(x)\rV^2-\omega^2\dist(x,K\cap U)^2\\&=(1-\omega)^2\lV x-P_{K\cap U}(x)\rV^2,
\end{align} 
which immediately implies the result.
\end{proof}

\Cref{p6} implies that if $\{x^k\}_{k\in\na}$ is the sequence generated by MAAP and $x^k\in V$ for large
enough $k$, then the sequence $\{\dist(x^k,K\cap U)\}_{k\in\na}$ converges Q-linearly, with asymptotic constant bounded above
by $\sqrt{1-\omega^2}$. In order to move from the distance sequence to the sequence $\{x^k\}_{k\in\na}$ itself, we will
invoke the following lemma from \cite{Arefidamghani:2021}.
 
\begin{lemma}\label{l2} 
Consider a nonempty closed convex set $M\subset\re^n$ and a sequence $\{y^k\}_{k\in\na}\subset\re^n$. Assume that $\{y^k\}_{k\in\na}$ is Fej\'er
monotone with respect to $M$, and that $\{\dist(y^k,M)\}_{k\in\na}$ converges R-linearly to $0$. Then $\{y^k\}_{k\in\na}$ converges 
R-linearly to some point $y^*\in M$, with asymptotic constant bounded above by the 
asymptotic constant of $\{\dist(y^k,M)\}_{k\in\na}$. 
\end{lemma}

\begin{proof}
See Lemma 3.4 in \cite{Arefidamghani:2021}.
\end{proof}

Next we establish the linear convergence of MAAP under \eqref{leb}.

\begin{theorem}\label{t2}
Consider a closed and convex set $K\subset\re^n$ and an affine manifold $U\subset\re^n,$ such that $K\cap U\neq \emptyset.$ Moreover, assume that $S$ is a separating operator 
for $K$ satisfying Assumptions \emph{A1--A3}. Suppose that $K,U$ and the separating operator $S$ satisfy \eqref{leb}. Let $\{x^k\}_{k\in\na}$
be the sequence generated by MAAP from any starting point $x^0\in\re^n$. If there exists $k_0$ such that
$x^k\in V$ for all $k\ge k_0$, then $\{x^k\}_{k\in\na}$ converges R-linearly to some point
$x^*\in K\cap U$, and the asymptotic constant is bounded above by $\sqrt{1-\omega^2}$, with $\omega$ and $V$ as in \eqref{leb}.
\end{theorem}
  
\begin{proof} 
The fact that $\{x^k\}_{k\in\na}$ converges to some $x^*\in K\cap U$ has been established in \Cref{t1}. 
Take any $k\ge k_0$. By assumption, $x^k\in V$ and by definition of $T^S$, $x^k\in U$.
So, we can take $x=x^k$ in \Cref{p6}, in which case $T^S(x)=x^{k+1}$, and rewrite \eqref{aa21}
as $(1-\omega^2)\dist(x^k,K\cap U)^2\ge \dist(x^{k+1},K\cap U)^2$  for $k\ge k_0$, which  implies 
that $\{\dist(x^k,K\cap U)\}_{k\in\na}$ converges Q-linearly,
and hence R-linearly,
with asymptotic constant bounded by $\sqrt{1-\omega^2}$. The corresponding result for the R-linear convergence
of $\{x^k\}_{k\in\na}$ with the same bound for the asymptotic constant follows then from \Cref{l2}, since $\{x^k\}_{k\in\na}$ is Fej\'er monotone
with respect to $K\cap U$ by \Cref{t1}. 
\end{proof} 

Now we analyze the convergence rate of CARM under \eqref{leb}, for which a preliminary result, relating $x, C^S(x)$ and $T^S(x)$, is needed.
The corresponding result for $x,C(x), T(x)$ can be found in \cite{Behling:2020}, where it is proved with a geometrical argument.
Here we present an analytical one.

\begin{proposition}\label{pqq}
Consider the operators $C^S,T^S:\re^n\to\re^n$ defined in \eqref{aa7}. Then $T^S(x)$ belongs to the segment between 
$x$ and $C^S(x)$ for all $x\in U$.
\end{proposition}

\begin{proof}
Let $E$ denote the affine manifold spanned by $x, R^S(x)$ and $R_U(R^S(x))$. By definition, the circumcenter of these
three points, namely $C^S(x)$, belongs to $E$. We claim that $T^S(x)$ also belongs to $E$, and next we proceed to prove the claim.
Since $U$ is an affine manifold, $P_U$ is an affine operator, so that $P_U(\alpha x+(1-\alpha)x')=\alpha P_U(x)+(1-\alpha)P_U(x')$
for all $\alpha\in\re$ and all $x,x'\in\re^n$. By \eqref{e1}, $R_U(R^S(x))=2P_U(R^S(x))-R^S(x)$, so that 
\begin{equation}\label{ei1}
P_U(R^S(x))=\frac{1}{2}\left(R_U(R^S(x))+R^S(x)\right).
\end{equation}
On the other hand, using the affinity of $P_U$, the definition of $T^S$ and the assumption that $x\in U$, we have
\begin{equation}\label{eii} 
P_U(R^S(x))=P_U(2P^S(x)-x)=2P_U(P^S(x))-P_U(x)=2T^S(x)-x,
\end{equation}
so that
\begin{equation}\label{ei2}
T^S(x)=\frac{1}{2}\left(P_U(R^S(x))+x\right).
\end{equation}
Combining \eqref{ei1} and \eqref{ei2},
\[ 
T^S(x)=\frac{1}{2}x+\frac{1}{4}R_U(R^S(x))+\frac{1}{4}R^S(x),
\] 
{\emph{i.e.}}, $T^S(x)$ is a convex combination of $x, R_U(R^S(x))$ and $R^S(x)$. Since these three points belong to $E$, the same
holds for $T^S(x)$ and the claim holds. We observe now that $x\in U$ by assumption, $T^S(x)\in U$ by definition,
and $C^S(x)\in U$ by \Cref{p5}(ii). Now we consider three cases: if dim$(E\cap U)=0$ then $x,T^S(x)$ and $C^S(x)$
coincide and the result holds trivially. If dim$(E\cap U)=2$ then $E\subset U$, so that $R^S(x)\in U$ so that
$R_U(R^S(x))=R^S(x)$, in which case $C^S(x)$ is the midpoint between $x$ and $R^S(x)$, which is precisely $P^S(x)$.
Hence, $P^S(x)\in U$, so that $T^S(x)=P_U(P^S(x))=P^S(x)=C^S(x)$, implying that $T^S(x)$ and $C^S(x)$ coincide, and the
result holds trivially. The interesting case is the remaining one, {\emph{i.e.}}, dim$(E\cap U)=1$. In this case $x, T^S(x)$ and
$C^S(x)$ lie in a line, so that we can write $C^S(x)=x+\eta(T^S(x)-x)$ with $\eta\in\re$, and it suffices to prove that $\eta\ge 1$.

By the definition of $\eta$, 
\begin{equation}\label{ei3}
\lV C^S(x)-x\rV=\lv\eta\rv\,\lV T^S(x)-x\rV.
\end{equation}
In view of \eqref{aa11} with $M=U, y=C^S(x)$ and $x=R^S(x)$,
\begin{equation}\label{ao}
\lV C^S(x)-R^S(x)\rV\ge\lV C^S(x)-P_U(R^S(x))\rV.
\end{equation}
Then
\begin{align}
\lV C^S(x)-x\rV&=\lV C^S(x)_R^S(x)\rV\ge\lV C^S(x)-P_U(R^S(x))\rV\\&=\lV\left(C^S(x)-x\right)-\left(P_U(R^S(x))-x\right)\rV\\&=
\lV\eta\left(T^S(x)-x\right)-2\left(T^S(x)-x\right)\rV\\&=\lv\eta-2\rv\,\lV T^S(x)-x\rV,\label{ei4}
\end{align}
using the definition of the circumcenter in the first equality, \eqref{ao} in the inequality, and 
\eqref{eii}, as well as the definition of $\eta$, in the third equality. Combining \eqref{ei3} and \eqref{ei4},
we get 
\[
\lv\eta\rv\,\lV T^S(x)-x\rV\ge\lv\eta-2\rv\,\lV T^S(x)-x\rV,
\]
implying that $\lv\eta\rv\ge\lv 2-\eta\rv$,
which holds only when $\eta\ge 1$, completing the proof.
\end{proof}
 
We continue with another intermediate result.
 
\begin{proposition}\label{p7}
Assume that \eqref{leb} holds for $K,U$ and $S$, and take $x\in U$. If $x,C^S(x)\in V$ then 
\begin{equation}\label{ae25}
(1+\omega^2)\dist(C^S(x),K\cap U)^2\le(1-\omega^2)\dist(x,K\cap U)^2,
\end{equation} 
with $V, \omega$ as in \eqref{leb}.
\end{proposition}

\begin{proof} Take $z\in K\cap U, x\in V\cap U$.
We use \Cref{p2}, rewriting \eqref{aa10}
as
\begin{equation}\label{aa25}
\lV x-P^S(x)\rV^2\le \lV x-z\rV^2-\lV P_U(P^S(x))-z\rV^2-\lV P_U(P^S(x))-P^S(x)\rV^2
\end{equation}
for all $x\in\re^n$ and all $z\in K\cap U$. Since $x\in U$, we get from \Cref{l1} that $C^S(x)=P_{H(x)}(x)$. We apply
next the well known characterization of projections~\cite[Theorem 3.16]{Bauschke:2017} to get
\begin{equation}\label{aa26}
(x-C^S(x))^{\top}(z-C^S(x))\le 0.
\end{equation} 
In view of \Cref{pqq}, 
$P_U(P^S(x))$ is a convex combination of $x$ and $C^S(x)$, meaning that $P_U(P^S(x))-C^S(x)$ is a nonnegative multiple
of $x-C^S(x)$, so that \eqref{aa26} implies 
\begin{equation}\label{aa27}
(P_U(P^S(x))-C^S(x))^{\top}(z-C^S(x))\le 0. 
\end{equation}
Expanding the inner product in \eqref{aa27}, we obtain
\begin{equation}\label{aa28}
\lV P_U(P^S(x))-z\rV^2\ge\lV C^S(x)-z\rV^2+\lV C^S(x)-P_U(P^S(x))\rV^2.
\end{equation}
Combining \eqref{aa25} and \eqref{aa28}, we have
\begin{align}
\dist(x,S(x))^2\le&\lV x-z\rV^2-\lV C^S(x)-z\rV^2-\lV C^S(x)-P_U(P^S(x))\rV^2\\&-\lV P_U(P^S(x))-P^S(x)\rV^2.\label{aa29}
\end{align}
Now, since $U$ is an affine manifold, $(y-P_U(y))^{\top}(w-P_U(y))=0$ for all $y\in\re^n$ and all $w\in U$,
so that 
\begin{equation}\label{aa30}
\lV w-y\rV^2=\lV w-P_U(y)\rV^2+\lV P_U(y)-y\rV^2.
\end{equation}  
Since $C^S(x)\in U$ by \Cref{l1}, we use \eqref{aa30} with $y=P^S(x), w=C^S(x)$, getting 
\begin{equation}\label{aa31}
\lV C^S(x)-P_U(P^S(x))\rV^2+\lV P_U(P^S(x))-P^S(x)\rV^2=\lV C^S(x)-P^S(x)\rV^2.
\end{equation}
Replacing \eqref{aa31} in \eqref{aa29}, we obtain
\begin{align}
\dist(x,S(x))^2&\le\lV x-z\rV^2-\lV C^S(x)-z\rV^2-\lV C^S(x)-P^S(x)\rV^2\\&
\le
\lV x-z\rV^2-\dist(C^S(x),K\cap U)^2-\dist(C^S(x),S(x))^2,\label{aa32}
\end{align}
using the facts that $P^S(x)\in S(x)$ and $z\in K\cap U$ in the last inequality.
Now, we take $z=P_{K\cap U}(x)$, recall that $x,C^S(x)\in V$ by hypothesis, and invoke the \eqref{leb} assumption,
together with \eqref{aa32}, in order to get
\begin{align}
\omega^2\dist(x,K\cap U)^2\le& \dist(x,S(x))^2\\\le& \dist(x,K\cap U)^2-\dist(C^S(x),K\cap U)^2\\&-\omega^2\dist(C^S(x),K\cap U)^2\\=&\dist(x,K\cap U)^2-(1+\omega^2)\dist(C^S(x),K\cap U)^2.\label{aa33} 
\end{align}
The result follows rearranging \eqref{aa33}.
\end{proof}

Next we present our convergence rate result for CARM.

\begin{theorem}\label{t3}
Consider a closed and convex set $K\subset\re^n$, an affine manifold $U\subset\re^n,$ such that $K\cap U\neq \emptyset$, and a separating operator $S$
for $K$ satisfying Assumptions \emph{A1--A3}. Suppose that $K,U$ and the separating operator $S$ satisfy \eqref{leb}. Let $\{x^k\}_{k\in\na}$
be the sequence generated by CARM from any starting point $x^0\in U$. If there exists $k_0$ such that
$x^k\in V$ for all $k\ge k_0$, then $\{x^k\}_{k\in\na}$ converges R-linearly to some point
$x^*\in K\cap U$, and the asymptotic constant is bounded above by $\sqrt{{(1-\omega^2)}/{(1+\omega^2)}}$, with $\omega$ and $V$ as in \eqref{leb}.
\end{theorem}
  
\begin{proof} 
The fact that $\{x^k\}_{k\in\na}$ converges to some $x^*\in K\cap U$ has been established in \Cref{t1}. 
Take any $k\ge k_0$. By assumption, $x^k\in V$ and by definition of $T^S$, $x^k\in U$. Also, $k+1\ge k_0$, so that 
$C^S(x^k)=x^{k+1}\in V$
So, we can take $x=x^k$ in \Cref{p7}, and rewrite \eqref{ae25}
as $(1+\omega^2)\dist(x^{k+1},K\cap U)^2\le(1-\omega^2)\dist(x^k,K\cap U)^2$ or equivalently
as 
\[
\frac{\dist(x^{k+1},K\cap U)}{\dist(x^k, K\cap U)}\le \sqrt{\frac{1-\omega^2}{1+\omega^2}},
\]  
for all $k\ge 0$, which immediately implies 
that $\{\dist(x^k,K\cap U)\}_{k\in\na}$ converges Q-linearly,
and hence R-linearly,
with asymptotic constant bounded by $\sqrt{(1-\omega^2)/(1+\omega^2)}$. The corresponding result for the R-linear convergence
of $\{x^k\}_{k\in\na}$ with the same bound for the asymptotic constant follows then from \Cref{l2}, since $\{x^k\}_{k\in\na}$ is Fej\'er monotone
with respect to $K\cap U$ by \Cref{t1}.
\end{proof} 

From now on, given $z\in\re^n, \alpha>0, B(z,\alpha)$ will denote the closed ball centered at $z$ with radius $\alpha$.

The results of \Cref{t2,t3} become relevant only if we are able to find a separating operator $S$ for $K$ such that
\eqref{leb} holds. This is only possible if the ``trivial'' separating operator 
satisfies an error bound, {\emph{i.e.}}, if an error bound holds for the sets $K,U$ themselves. Let $\{ x^k\}_{k\in\na}$ be a sequence generated by
CARM starting at some $x^0\in U$. By \Cref{t1}, $\{x^k\}_{k\in\na}$ converges to some $x^*\in K\cap U$. Without loss of generality, we assume 
that $x^k\notin K$ for all $k$, because otherwise the sequence is finite and it makes no sense to deal with convergence rates.
In such a case, $x^*\in \partial K$, the boundary of $K$. We also assume from 
now on that a local error bound for $K,U,$ say \eqref{leb1}, holds at some neighborhood of $x^*$, {\emph{i.e.}}

\begin{enumerate}[label=LEB1),font=\bfseries,leftmargin=*,ref=LEB1] 
\item \label[assumption]{leb1} There exist $\rho,\bar\omega>0$ such that $\dist(x,K)\ge\bar\omega \dist(x,K\cap U)$ for all
$x\in U\cap B(x^*,\rho)$.
\end{enumerate}

Note that \eqref{leb1} is simply a local version of \eqref{eb}. 
Observe further that \eqref{leb1} does not involve the separating operator $S$, and that it gives a specific form to the set $V$ in \eqref{leb}, namely a ball around $x^*$.

We will analyze the situation for what we call the ``prototypical'' separating operator, namely the operator $S$ presented
in \Cref{ex1}, for the case in which $K$ is the $0$-sublevel set of a convex function $g:\re^n\to\re$. 

We will prove that under some additional mild assumptions on $g$,
for any $\omega <\bar\omega$ there exists a set $V$ such that $U,K,S$ satisfy a local error bound assumption, say \eqref{leb},
with the pair $\omega, V$. 

 Indeed, it will not be necessary to assume \eqref{leb} in the convergence rate result; we will prove that under \eqref{leb1}, \eqref{leb} will be satisfied for any $\omega<\bar\omega$ with an appropriate set $V$ which does
contain the tail of the sequence. 

Our proof strategy will be as follows: in order to check that the error bounds for $K,U$ and $S(x),U$ are virtually the same
for $x$ close to the limit $x^*$ of the CARM sequence, we will prove that the quotient between $\dist(x,S(x))$ and $\dist(x,K)$ approaches
$1$ when $x$ approaches $x^*$. Since both distances vanish at $x=x^*$, we will take the quotient of their first order approximations,
in a L'H\^ospital's rule fashion, and the result will be established, as long as the numerator and denominator of the new quotient
are bounded away from $0$, because otherwise this quotient remains indeterminate. This ``bad'' situation occurs when $x$ approaches
$x^*$ along a direction almost tangent to $K\cap U$, or equivalently almost normal to $\nabla g(x^*)$. Fortunately, the CARM 
sequence, being Fej\'er monotone with respect to $K\cap U$, does not approach $x^*$ in such a tangential way. We will take
an adequate value smaller than the angle between $\nabla g(x^*)$ and $x^k-x^*$. Then, we will exclude directions whose angle with
$\nabla g(x^*)$ is smaller than such value, and find a ball around $x^*$ such that, given any $\omega<\bar\omega$, \eqref{leb}
holds with parameter $\omega$ in the set $V$ defined as the ball minus the ``slice'' containing the ``bad'' directions. Because of the
Fej\'er monotonicity of the CARM sequence, its iterates will remain in $V$ for large enough $k$, and the results of \Cref{t3}
will hold with such $\omega$. We proceed to follow this strategy in detail.           

The additional assumptions on $g$ are continuous differentiability and a Slater condition, meaning that there exists 
$\hat x\in\re^n$ such that $g(\hat x)<0$. When $g$ is of class ${\cal C}^1$, the separating operator of \Cref{ex1}
becomes
\begin{equation}\label{aa34}
S(x)=\begin{cases} K, & {\rm if}\,\,x\in K\\
\{z\in\re^n:\nabla g(x)^{\top}(z-x)+g(x)\leq 0\}\, & {\rm otherwise}.
\end{cases} 
\end{equation}

\begin{proposition}\label{p8}
Let $g:\re^n\to\re$ be convex, of class ${\cal C}^1$ and such that there exists $\hat x\in\re^n$ satisfying
$g(\hat x)< 0$. Take $K=\{x\in\re^n: g(x)\le 0\}$. Assume that $K,U$ satisfy \eqref{leb1}. Take $x^*$ as in \eqref{leb1}, 
fix $0<\nu <\lV\nabla g(x^*)\rV$ (we will establish that $0\ne\nabla g(x^*)$ in the proof of this proposition), 
and define $L_\nu:=\{z\in\re^n: \lv\nabla g(x^*)^{\top}(z-x^*)\rv\le\nu\lV z-x^*\rV\}$. Consider the separating
operator $S$ defined in \eqref{aa34}. Then, for any $\omega<\bar\omega$, with
$\bar\omega$ as in \eqref{leb1}, there exists $\beta>0$ such that $K,U,S$ satisfy \eqref{leb} with 
$\omega$ and $V:=B(x^*,\beta)\setminus L_\nu$.
\end{proposition}    

\begin{proof}
The fact that $0 < \nu<\lV\nabla g(x^*)\rV$ ensures that $V\ne\emptyset$.
We will prove that for $x$ close to $x^*$ the quotient $\dist(x,S(x))/\dist(x,K)$ approaches $1$,  
and first we proceed to evaluate $\dist(x,S(x))$. Note that when $x\in K\subset S(x)$, the inequality 
in \eqref{leb1} holds trivially because of A1. Thus, we assume that $x\notin K$, so that $x\notin S(x)$ 
by \Cref{pq1}, and hence $g(x)> 0$ and 
$S(x)=\{z\in\re^n:\nabla g(x)^{\top}(x-z)+g(x)\}$, implying, in view of \eqref{ee1}, that
\begin{equation}\label{aa35} 
\dist(x,S(x))=\lV x-P^S(x)\rV=\frac{g(x)}{\lV\nabla g(x)\rV}.
\end{equation}
Now we look for a more manageable expression for $\dist(x,K)=\lV x-P_K(x)\rV$. Let $y=P_K(x)$. So, $y$ is the unique solution of the
problem $\min\lV z-x\rV^2$ s.t. $g(z)\le 0$, whose first order optimality conditions, sufficient by convexity of $g$,
are 
\begin{equation}\label{ae35}
x-z=\lambda\nabla g(z) 
\end{equation}
with $\lambda\ge 0$, so that 
\begin{equation}\label{aa36} 
\dist(x,K)=\lV x-y\rV=\lambda\lV\nabla g(y)\rV. 
\end{equation}
Now we observe that the Slater condition implies that the right hand sides of both \eqref{aa35} and \eqref{aa36} are well defined:  
since $x\notin K$, $g(x)> 0$; since $y=P_K(x)\in\partial K$, $g(y)=0$. By the Slater condition, $g(x)>g(\hat x)$ and $g(y)>g(\hat x)$,
so that neither $x$ nor $y$ are minimizers of $g$, and hence both $\nabla g(y)$ and $\nabla g(x)$ are nonzero. By the same token, 
$\nabla g(x^*)\ne 0$, because $x^*$ is not a minimizer of $g$: being the limit of a sequence lying
outside $K$, $x^*$ belongs to the boundary of $K$, so that $g(x^*)=0>g(\hat x)$.  
 
From \eqref{aa35}, \eqref{aa36},
\begin{equation}\label{aa37}
\frac{\dist(x,S(x))}{\dist(x,K)}=\lV\nabla g(y(x))\rV\,\lV\nabla g(x)\rV\left[\frac{\lambda(x)}{g(x)}\right],
\end{equation}
where the notation $y(x),\lambda(x)$ emphasizes that both $y=P_K(x)$ and the multiplier $\lambda$ depend on $x$.

We look at the right hand side \eqref{aa37} for $x$ close to $x^*\in K$, in which case $y$, by continuity of $P_K$, approaches
$P(x^*)=x^*$, so that $\nabla g(y(x))$ approaches $\nabla g(x^*)\ne 0$, and hence, in view of \eqref{aa35}, 
$\lambda(x)$ approaches $0$. So, the product of the first two factors in the right hand side of \eqref{aa37} approaches $\lV\nabla g(x^*)\rV^2$, 
but the quotient is indeterminate, because both the numerator and the denominator approach $0$, requiring a more precise 
first order analysis. 

Expanding $g(x)$ around $x^*$ and taking into account that $g(x^*)=0$, we get 
\begin{equation}\label{aa38}
g(x)=\nabla g(x^*)(x-x^*)+o(\lV x-x^*\rV).
\end{equation}
Define $t=\lV x-x^*\rV, d=t^{-1}(x-x^*)$ so that $\lV d\rV=1$, and \eqref{aa39} becomes 
\begin{equation}\label{aa39}
g(x)=t\nabla g(x^*)^{\top}d+o(t).
\end{equation}
No we look at $\lambda(x)$. Let $\phi(t)=\lambda(x^*+td)$. Note that for $x\in\partial K$ we get $y(x)=x$, so that
$0=\lambda(x)\nabla g(x)$ and hence $\lambda(x)=0$. Thus, $\phi(0)=0$ and     
\begin{equation}\label{aa40}
\lambda(x)=\phi(t)=t\phi'_+(0)+o(t),
\end{equation}
where $\phi'_+(0)$ denotes the right derivative of $\phi(t)$ at $0$.
Since we assume that $x\notin K$, we have $y(x)\in\partial K$ and hence, using \eqref{ae35},  
\begin{equation}\label{aa41}
0=g(y(x))=g(x-\lambda(x)\nabla g(y(x)))=g(x^*+td-\phi(t)\nabla g(y(x^*+td)))
\end{equation}
for all $t>0$. Let $\sigma(t)=\phi(t)\nabla g(y(x^*+td))$, $\psi(t)=g(x^*+td-\sigma(t))$, so that \eqref{aa41} becomes
$0=\psi(t)=g(x^*+td-\sigma(t))$ for all $t>0$ and hence
\begin{equation}\label{aa42}
0=\psi'(t)=\nabla g(y(x^*+td))^{\top}(d-\sigma'(t))
\end{equation}
Taking limits in \eqref{aa42} with $t\to 0^+$, and noting that $y(x^*)=x^*$ because $x^*\in K$, we get
\begin{equation}\label{aa43}
0= \nabla g(x^*)^{\top}(d-\sigma'_+(0)),
\end{equation}
where $\sigma'_+(0)$ denotes the right derivative of $\sigma(t)$ at $0$.
We compute $\sigma'_+(0)$ directly from the definition, because we assume 
that $g$ is of class ${\cal C}^1$ but perhaps not of class ${\cal C}^2$.
Recalling that $\phi(0)=0$, we have that
\begin{equation}\label{aa44}
\sigma'_+(0)=\lim_{t\to 0^+}\frac{\phi(t)}{t}\nabla g(y(x^*+td))=
\lim_{t\to 0^+}\frac{\phi(t)}{t}\lim_{t\to 0^+}\nabla g(y(x^*+td))=\phi'_+(0)\nabla g(x^*),
\end{equation}
using the facts that $g$ is class ${\cal C} ^1$ and that $y(x^*)=x^*$. Replacing \eqref{aa44} in \eqref{aa43},
we get that $0=\nabla g(x^*)^{\top}(d-\phi'_+(0)\nabla g(x^*))$, and therefore 
\begin{equation}\label{aa45}
\phi'_+(0)=\frac{\nabla g(x^*)^{\top}d}{\lV\nabla g(x^*)\rV^2}.
\end{equation}
Using \eqref{aa40} and \eqref{aa45} we obtain
\begin{equation}\label{aa46}
\lambda(x)=\frac{t\nabla g(x^*)^{\top}d}{\lV\nabla g(x^*)\rV^2}+o(t)=\frac{1}{\lV\nabla g(x^*)\rV^2}[t\nabla g(x^*)^{\top}d +o(t)].
\end{equation}
Replacing \eqref{aa46} and \eqref{aa39} in \eqref{aa37}, we obtain
\begin{align}
\hspace*{-2em}\frac{\dist(x,S(x))}{\dist(x,K)}&=
\left[\frac{\lV\nabla g(y(x))\rV\,\lV\nabla g(x)\rV}{\lV \nabla g(x^*)\rV^2}\right]
\left[\frac{t\nabla g(x^*)^{\top}d +o(t)}{t\nabla g(x^*)^{\top}d +o(t)}\right]\\
&=
\left[\frac{\lV\nabla g(y(x^*+td))\rV\,\lV\nabla g(x^*+td)\rV}{\lV \nabla g(x^*)\rV^2}\right]
\left[\frac{\nabla g(x^*)^{\top}d +o(t)/t}{\nabla g(x^*)^{\top}d +o(t)/t}\right].\label{aa47}
\end{align}
Now we recall that we must check the inequality of \eqref{leb} only for points in $V$, and that $V\cap L_\nu=\emptyset$,
with $L_\nu=\{z\in \re^n:\nabla g(x^*)(z-x^*)\le\nu\lV z-x^*\rV\}$. 
So, for $x\in V$ we have $\lv\nabla g(x^*)^{\top}(x-x^*)\rv\ge\nu\lV x-x^*\rV$, which implies
$\lv\nabla g(x^*)^{\top}d\rv\ge\nu$, {\emph{i.e.}}, $\nabla g(x^*)^{\top}d$ is bounded away from $0$, independently of the direction $d$. 
In this situation, it is clear that the rightmost expression of \eqref{aa47} tends to $1$
when $t\to 0^+$, and so there exists some $\beta >0$ such that for $t\in (0,\beta)$ such expression is not smaller
than $\omega/\bar\omega$, with $\omega$ as in \eqref{leb} and $\bar\omega$ as in \eqref{leb1}. Without loss of generality,
we assume that $\beta\le\rho$, with $\rho$ as in Assumption \eqref{leb1}. Since $t=\lV x-x^*\rV$,
we have proved that for $x\in U\cap B(x^*,\beta)\setminus L_\nu=U\cap V$ it holds that
\begin{equation}\label{aa48}   
\frac{\dist(x,S(x))}{\dist(x,K)}\ge\frac{\omega}{\bar\omega}.
\end{equation}
It follows from \eqref{aa48} that 
\begin{equation}\label{aa49}
\dist(x,S(x))\ge \dist(x,K)\frac{\omega}{\bar\omega} 
\end{equation}
for all $x\in V\cap U$. Dividing both sides of \eqref{aa49} by $\dist(x,K\cap U)$, recalling that $\beta\le\rho$, and 
invoking Assumption \eqref{leb1}, we obtain 
\[
\frac{\dist(x,S(x))}{\dist(x,K\cap U)}\ge\frac{\dist(x,K)}{\dist(x,K\cap U)}\frac{\omega}{\bar\omega}\ge\bar\omega \frac{\omega}{\bar\omega}=\omega
\]
for all $x\in U\cap V$, thus proving that \eqref{leb} holds for any $\omega<\bar\omega$, with $V=B(x^*,\beta)\setminus L_\nu$, where
$\bar\omega$ is the local error bound for the sets $K,U$. 
\end{proof}

We have proved that for the prototypical separating operator given by \eqref{aa34}, the result of \Cref{p7} holds. 
In order to
obtain the convergence rate result of \Cref{t3} for this operator, we must prove that in this case the tail of the sequence 
$\{x^k\}_{k\in\na}$ generated
by CARM is contained in $V=B(x^*,\beta)\setminus L_\nu$. Note that $\beta$ depends on $\nu$. Next we will show that
if we take $\nu$ smaller than a certain constant which depends on $x^*$, the initial iterate $x^0$, the Slater point $\hat x$ 
and the parameter $\bar\omega$ of \eqref{leb1}, 
then the tail of the sequence $\{x^k\}_{k\in\na}$ will remain outside $L_\nu$. Clearly, this will suffice, 
because the sequence
eventually remains in any ball around its limit, which is $x^*$, so that its tail will surely be contained in $B(x^*,\beta)$. 
The fact that $x^k\notin L_\nu$ for large enough $k$ is
a consequence of the Fej\'er monotonicity of the sequence with respect to $K\cap U$, proved in \Cref{t1}.  
In the next proposition we will
prove that indeed $x^k\notin L_\nu$ for large enough $k$,
and so the result of \Cref{t3} holds for this separating operator.  

\begin{proposition}\label{p9} 
Let $g:\re^n\to\re$ be convex, of class ${\cal C}^1$ and such that there exists $\hat x\in\re^n$ satisfying
$g(\hat x)< 0$. Take $K=\{x\in\re^n: g(x)\le 0\}$. Assume that $K,U$ satisfy \eqref{leb1}. 
Consider the separating
operator $S$ defined in \eqref{aa34}. Let $\{x^k\}_{k\in\na}$ be a sequence generated by (CARM) with starting 
point $x^0\in U$ and limit point
$x^*\in K\cap U$. Take $\nu >0$ satisfying
\begin{equation}\label{eio}
\nu <\min\left\{\frac{\bar\omega\lv g(\hat x)\rv}{4\left(\lV\hat x-x^*\rV+\lV x^*-x^0\rV\right)},\frac{\lV\nabla g(x^*)\rV}{2}\right\},
\end{equation}
with $\bar\omega$ as in \eqref{leb1}, and define 
\[
L_\nu:=\{z\in\re^n:\lv\nabla g(x^*)^{\top}(z-x^*)\rv\le\nu\lV z-x^*\rV\}.
\] 
Then, there exists $k_0$ such that for all
$k\ge k_0$, $x_k\in B(x^*,\beta)\setminus L_\nu$, with $\beta$ as in \Cref{p8}.
\end{proposition}

\begin{proof}
Assume that $x^k\in L_\nu$, {\emph{i.e.}},
\begin{equation}\label{aa50}
\lv\nabla g(x^*)^{\top}(x^k-x^*)\rv\le\nu\lV x^k-x^*\rV. 
\end{equation} 
Using the gradient inequality, the fact that $g(x^*)=0$ and \eqref{aa50}, we obtain
\begin{align}
g(x^k) & \le g(x^*)-\nabla g(x^k)^{\top}(x^*-x^k) \\
& =[\nabla g(x^*)-\nabla g(x^k) -\nabla g(x^*)]^{\top}(x^*-x^k)
\\
& \leq \lV\nabla g(x^*)-\nabla g(x^k)\rV\,\lV x^*-x^k\rV +\lv\nabla g(x^*)^{\top}(x^k-x^*)\rv \\
& \le
\left(\lV\nabla g(x^*)-\nabla g(x^k)\rV+\nu\right)\lV x^k-x^*\rV.\label{aa51}
\end{align}

By \Cref{t1}, $\{x^k\}_{k\in\na}$ is Fej\'er monotone with respect to $K\cap U$. Thus, we use \Cref{p1}(iii) and \eqref{leb1} in
\eqref{aa51}, obtaining
\begin{align}
g(x^k)& \le 
2\left(\lV\nabla g(x^*)-\nabla g(x^k)\rV+\nu\right)\dist(x^k,K\cap U)\\& \le\frac{2\left(\lV\nabla g(x^*)-\nabla g(x^k)\rV+\nu\right)\dist(x^k,K)}{\bar\omega}.\label{aa52}
\end{align}
Denote $y^k=P_K(x^k)$. Using again the gradient inequality, together with the facts that $g(y^k)=0$ and that $x^k-y^k$ and $\nabla g(y^k)$
are collinear, which is a consequence of \eqref{ae35} and the nonnegativity of $\lambda$, we get from \eqref{aa52}
\begin{align}
g(x^k)& \ge g(y^k)+\nabla g(y^k)^{\top}(x^k-y^k)\\ 
&=\lV\nabla g(y^k)\rV\,\lV x^k-y^k\rV=\lV\nabla g(y^k)\rV \dist(x^k,K).\label{aa53}
\end{align}
Now we use the Slater assumption on $g$ for finding a lower bound for $\lV\nabla g(y^k)\rV$. Take $\hat x$ such that $g(\hat x) <0$,
and apply once again the gradient inequality.
\begin{equation}\label{aa54}
g(\hat x)\ge g(y^k)+\nabla g(y^k)^{\top}(\hat x-y^k)=\nabla g(y^k)^{\top}(\hat x-y^k)\ge-\lV\nabla g(y^k)\rV\,\lV\hat x-y^k\rV.
\end{equation}
Multiplying \eqref{aa54} by $-1$, we get
\begin{align}
\lv g(\hat x)\rv & \le\lV\nabla g(y^k)\rV\,\lV\hat x-y^k\rV\le\lV\nabla g(y^k)\rV\left(\lV\hat x-x^*\rV+\lV x^*-y^k\rV\right)
\\ & \le
\lV\nabla g(y^k)\rV\left(\lV\hat x-x^*\rV+\lV x^*-x^k\rV\right)\\ 
&\le\lV\nabla g(y^k)\rV\left(\lV\hat x-x^*\rV+\lV x^*-x^0\rV\right),\label{aa55}
\end{align}
using the facts that $y^k=P_K(x^k)$ and that $x^*\in K$ in the third inequality and the F\'ejer monotonicity of $\{ x^k\}_{k\in\na}$ with 
respect to $K\cap U$ in the fourth one.
Now, since $\lim_{k\to\infty} x^k=x^*$, 
there exists $k_1$ such that $\lV x^k-x^*\rV\le\rho$ for $k\ge k_1$, with $\rho$ as in \eqref{leb1}. 
So, in view of \eqref{aa55}, with $k\ge k_1$,
$\lv g(\hat x)\rv\le\lV\nabla g(y^k)\rV(\lV\hat x-x^*\rV+\lV x^*-x^0\rV)$, implying that
\begin{equation}\label{aa56}
\lV\nabla g(y^k)\rV\ge\frac{\lv g(\hat x)\rv}{\lV\hat x-x^*\rV+\lV x^*-x^0\rV}.
\end{equation}
Combining \eqref{aa52}, \eqref{aa53}, \eqref{aa56} and \eqref{eio}, we obtain
\[
2\nu<\frac{\bar\omega\lv g(\hat x)\rv}{2\left(\lV\hat x-x^*\rV+\lV x^*-x^0\rV\right)}\le\lV\nabla g(x^k)-\nabla g(x^*)\rV+\nu,
\]
implying
\begin{equation}\label{aa57}
\nu <\lV\nabla g(x^k)-\nabla g(x^*)\rV.
\end{equation}
 
The inequality in \eqref{aa57} has been obtained by assuming that $x^k\in L_\nu$. Now, since $\lim_{k\to\infty}x^k=x^*$ and
$g$ is of class ${\cal C}^1$, there exists $k_0\ge k_1$ such that $\lV\nabla g(x^*)-\nabla g(x^k)\rV\le\nu$ for
$k\ge k_0$, and hence \eqref{aa57} implies that for $k\ge k_0$, $x^k\notin L_\nu$. Since $k_0\ge k_1$, $x^k\in B(x^*,\beta)$
for $k\ge k_0$, meaning that when $k\ge k_0$, $x^k\in B(x^*,\beta)\setminus L_\nu$, establishing the result.
\end{proof}

Now we conclude the analysis of CARM with the prototypical separating operator, proving that under smoothness of $g$ and a 
Slater condition, the CARM method achieves linear convergence with precisely the same bound for the asymptotic constant
as CRM, thus showing that the approximation of $P_K$ by $P^S$ produces no deterioration in the convergence rate. We emphasize
again that for this  operator $S$, $P_S$ has an elementary closed formula, namely the one given by 
\[
P^S(x)= x-\left(\frac{\max\{0,g(x)\}}{\lV\nabla g(x)\rV^2}\right)\nabla g(x). 
\]

\begin{theorem}\label{t4}
Let $g:\re^n\to\re$ be convex, of class ${\cal C}^1$ and such that there exists $\hat x\in\re^n$ satisfying
$g(\hat x)< 0$. Take $K=\{x\in\re^n: g(x)\le 0\}$. Assume that $K,U$ satisfy \eqref{leb1}. 
Consider the separating
operator $S$ defined in \eqref{aa34}. Let $\{x^k\}_{k\in\na}$ be a sequence generated by CARM with starting 
point $x^0\in U$. Then $\{x^k\}_{k\in\na}$ converges to some $x^*\in K\cap U$ with linear convergence rate, and
asymptotic constant bounded above by $\sqrt{{(1-\bar\omega^2)}/{(1+\bar\omega^2)}}$, with $\bar\omega$ as in \eqref{leb1}.
\end{theorem}

\begin{proof}
The fact that $\{x^k\}_{k\in\na}$ converges to some $x^*\in K\cap 1$ follows from \Cref{t1}. Let $\bar\omega$ be the parameter
in \eqref{leb1}.
By \Cref{p8},
$P,K$ and $S$ satisfy \eqref{leb} with any parameter $\omega\le\bar\omega$ and a suitable $V$. By \Cref{p9} $x^k\in V$ for large 
enough $k$, so that the assumptions of \Cref{t3} hold, and hence
\begin{equation}\label{aa58}
\limsup_{k\to\infty}\frac{\lV x^{k+1}-x^*\rV}{\lV x^k-x^*\rV}\le\sqrt{\frac{1-\omega^2}{1+\omega^2}}
\end{equation}
for any $\omega\le\bar\omega$. Taking infimum in the right hand side of \eqref{aa58} with $\omega <\bar\omega$, 
we conclude that the inequality holds also for $\bar\omega$, {\emph{i.e.}}
\[
\limsup_{k\to\infty}\frac{\lV x^{k+1}-x^*\rV}{\lV x^k-x^*\rV}\le\sqrt{\frac{1-\bar\omega^2}{1+\bar\omega^2}},
\] 
completing the proof.
\end{proof}

We mention that the results of \Cref{p8,p9} and \Cref{t4} can be extended without any complications to 
the separating
operator $\mathbf{S}$ in \Cref{ex2}, so that they can be applied for accelerating SiPM for CFP with $m$ convex sets, presented 
as $0$-sublevel sets of smooth convex functions. We omit the details.

Let us continue with a comment on the additional assumptions on $g$ used for proving \Cref{t4}, namely continuous differentiability
and the Slater condition. We guess that the second one is indeed needed for the validity of the result; regarding smoothness of $g$,
we conjecture that the CARM sequence still converges linearly under \eqref{leb} when $g$ is not smooth, but with an asymptotic constant
possibly larger than the one for CRM. It seems clear that the proof of such result requires techniques quite different from those used here. 

Finally, we address the issue of the extension of the results in this paper to the framework of infinite dimensional Hilbert spaces. We have
refrained from developing our analysis in such a framework because our
main focus lies in the extension of the convergence rate results for the exact algorithms presented in \cite{Arefidamghani:2021} to the approximate methods introduced in this paper, so that in order to establish the appropriate comparisons between the exact and approximate
methods one should start by rewriting the results of  \cite{Arefidamghani:2021} in the context of Hilbert spaces, which would 
unduly extend the length of this paper. We just comment that it is possible to attain such aim, following the approach presented 
in \cite{Bauschke:2021c, Bauschke:2020c}. 

\section{Convergence rate results for CARM and MAAP applied to specific instances of CFP}\label{s4}

The results of \Cref{s3} indicate that when $K,U$ satisfy an error bound assumption, 
both CARM  and MAAP enjoy linear convergence rates (with a better asymptotic constant for the 
former). In this section we present two families of CFP instances 
for which the difference between CARM and MAAP is more dramatic: using the prototypical separating operator, 
in the first one (for which \eqref{leb}
does not hold), MAAP converges sublinearly and CARM converges linearly; in the second one, MAAP converges linearly, 
as in \Cref{s3}, but CARM converges superlinearly. Similar results on the behavior of MAP and CRM 
for these two families can be found in \cite{Arefidamghani:2021}.

Throughout this section, $K\subset\re^{n+1}$ will be the epigraph of a convex function $f:\re^n\to\re$ of class 
${\cal C}^1$ and $U$ will be the hyperplane $U:=\{x\in\re^{n+1}:x_{n+1}=0\}$. We mention that the specific form of $U$ 
and the fact that $K$ is an epigraph entail little loss of generality; but the smoothness assumption on $f$ and the
fact that $U$ is a hyperplane (\emph{i.e.} an affine manifold of codimension $1$), are indeed more restrictive.

First we look at the case when the following assumptions hold:
\begin{listi}
\item[B1.] $f(0)=0$.
\item[B2.] $\nabla f(x)=0$ if and only if $x=0$.
\end{listi} 
Note that under B1--B2, $0$ is the unique minimizer of $f$ and that $K\cap U=\{0\}$. It follows from \Cref{t1}
that the sequences generated by MAAP and CARM, from any initial iterate in $\re^n$ and $U$ respectively, converge to $x^*=0$.
We prove next that under these assumptions MAAP converges sublinearly.

\begin{proposition}\label{p10} 
Assume that $K\subset\re^{n+1}$ is the epigraph of a convex function $f:\re^n\to\re$ 
of class ${\cal C}^1$ satisfying
B1--B2, and $U:=\{x\in\re^{n+1}:x_{n+1}=0\}$. Consider the separating operator given by \eqref{aa34} for the function $g:\re^{n+1}\to\re$
defined as $g(x_1, \dots ,x_{n+1})=f(x_1, \dots ,x_n)-x_{n+1}$. Then the sequence $\{x^k\}_{k\in\na}$ generated by MAAP starting at any
$x^0\in\re^{n+1}$ converges sublinearly to $x^*=0$.
\end{proposition}
\begin{proof}
Convergence of $\{x^k\}_{k\in\na}$ to $x^*=0$ results from \Cref{t1}. 
We write vectors in $\re^{n+1}$ as $(x,s)$ with $x\in\re^n,s\in\re$. We start by computing the formula for $T^S(x,0)$. 
By definition of $g$, $\nabla g(x,s)=(\nabla f(x),-1)^{\top}$. Let 
\begin{equation}\label{uu1}
\alpha(x) =\lV\nabla f(x)\rV^2+1.
\end{equation} 
By \eqref{ee1}, 
\[
P^S(x,0)=(x,0)-\frac{g(x,0)}{\lV\nabla g(x,0)\rV^2}\nabla g(x,0)= \left(x-\frac{f(x)}{\alpha(x)}\nabla f(x),-\frac{f(x)}{\alpha(x)}\right),
\]
which implies, since $P_U(x,s)=(x,0)$,
\begin{equation}\label{u1}
T^S(x,0)=P_U(P^S(x))=\left(x-\frac{f(x)}{\alpha(x)}\nabla f(x),0\right)
\end{equation}
Let $\bar x=\lV x\rV^{-1}x$. From \eqref{u1},
\begin{equation}\label{u2}
\left[\frac{\lV T^S(x,0)\rV}{\lV(x,0)\rV}\right]^2=
1-2\frac{f(x)}{\lV x\rV}\left(\frac{\nabla f(x)^{\top}\bar x}{\alpha(x)}\right)+\left(\frac{f(x)}{\lV x\rV}\frac{\lV\nabla f(x)\rV}{\alpha(x)}\right)^2.
\end{equation}
Note that $\lim_{x\to 0}\alpha (x)=\alpha(0)=1$ and that, by B1--B2, $\lim_{x\to 0}\nabla f(x)=\nabla f(0)=0$, $f(x)=o(\lV x\rV)$,
implying that $\lim_{x\to 0}f(x)/\lV x\rV=0$, and conclude from \eqref{u2} that
\begin{equation}\label{u3} 
\lim_{x\to 0}\frac{\lV T^S(x,0)\rV}{\lV(x,0)\rV}=1.
\end{equation}
Now, since $x^{k+1}=T^S(x^k)$, $x^k\in U$ for all $k\ge 0$, and $x^*=0$, we get from \eqref{u3}
\[
\lim_{k\to\infty}\frac{\lV x^{k+1}-x^*\rV}{\lV x^k-x^*\rV}=\lim_{x\to 0}\frac{\lV T^S(x,0)\rV}{\lV(x,0)\rV}=1,
\]   
and hence $\{x^k\}_{k\in\na}$ converges sublinearly.
\end{proof}

Next we study the CARM sequence in the same setting.

\begin{proposition}\label{p11} 
Assume that $K\subset\re^{n+1}$ is the epigraph of a convex function $f:\re^n\to\re$ 
of class ${\cal C}^1$ satisfying
\emph{B1--B2}, and $U:=\{x\in\re^{n+1}:x_{n+1}=0\}$. Consider the separating operator given by \eqref{aa34} for the function $g:\re^{n+1}\to\re$
defined as $g(x_1, \dots ,x_{n+1})=f(x_1, \dots ,x_n)-x_{n+1}$.
For $0\ne x\in\re^n$, define 
\begin{equation}\label{u4}
\theta(x):=\frac{f(x)}{\lV x\rV\,\lV\nabla f(x)\rV}.
\end{equation}
Then
\begin{equation}\label{u5}
\left[\frac{\lV C^S(x,0)\rV}{\lV x\rV}\right]^2\le 1-\theta(x)^2,
\end{equation}
with $C^S$ as in \eqref{aa7}.
\end{proposition}

\begin{proof} Define
\begin{equation}\label{uu5} 
\beta(x):=\frac{f(x)}{\lV\nabla f(x)\rV^2}.
\end{equation}
By \eqref{aa7},  
\begin{equation}\label{u6}
R^S(x,0)=\left(x-2\frac{f(x)}{\alpha(x)}\nabla f(x), 2\frac{f(x)}{\alpha(x)}\right).
\end{equation}
From \Cref{pqq}, 
\begin{equation}\label{u7}
C^S(x,0)=(x,0)+\eta(T^S(x,0)-(x,0))=\left(x-\eta\frac{f(x)}{\alpha(x)}\nabla f(x),0\right),
\end{equation}
for some $\eta\ge 1$.
By the definition of circumcenter, $\lV C^S(x)-x\rV=\lV C^S(x)-R^S(x)\rV$. Combining
this equation with \eqref{u6} and \eqref{u7}, one obtains $\eta=1+\lV\nabla f(x)\rV^{-1}$,
which implies, in view of \eqref{uu5}, that 
\begin{equation}\label{u8}
\frac{\eta f(x)}{\alpha(x)}=\frac{f(x)}{\lV\nabla f(x)\rV^2}=\beta(x).
\end{equation}
Combining \eqref{u7} and \eqref{u8},
\begin{equation}\label{uu8}
C^S(x,0)=(x-\beta(x)\nabla f(x),0),
\end{equation}
so that
\begin{align}
\lV C^S(x,0)\rV^2 & =\lV x\rV^2-2\beta(x)\nabla f(x)^{\top}x+\beta(x)^2\lV\nabla f(x)\rV^2 \\
& \le
\lV x\rV^2-2\beta(x)f(x)+\beta(x)^2\lV\nabla f(x)\rV^2,\label{u9}
\end{align}
using the fact that $f(x)\le\nabla f(x)^{\top}x$, which follows from that gradient inequality with
the points $x$ and $0$. It follows from \eqref{u9} and the definitions of $\alpha(x), \beta(x)$, that
\begin{align}
\left[\frac{\lV C^S(x,0)\rV}{\lV x\rV}\right]^2& \le 
1-2\beta(x)\frac{f(x)}{\lV x\rV^2}+\left(\frac{\beta(x)\lV\nabla f(x)\rV}{\lV x\rV}\right)^2
\\ & =
\label{u10}
1-\left(\frac{f(x)}{\lV\nabla f(x)\rV\,\lV x\rV}\right)^2=1-\theta(x)^2,
\end{align}
using \eqref{u4} in the last equality.
\end{proof}
 We prove next the linear convergence of the CARM sequence in this setting under the following additional assumption on $f$:

\begin{listi}

\item[B3)]
$\displaystyle
\liminf_{x\to 0}\frac{f(x)}{\lV x\rV\,\lV\nabla f(x)\rV}>0.
$

 \end{listi}
\begin{corollary}\label{c1}
Under the assumptions of \Cref{p11}, if $f$ satisfies \emph{B3} and $\{x^k\}_{k\in\na}$ is the sequence generated by CARM starting 
at any $x^0\in U$, then $\lim_{k\to\infty}x^k=x^*=0$, and
\[
\liminf_{k\to\infty}\frac{\lV x^{k+1}-x^*\rV}{\lV x^k-x^*\rV}\le\sqrt{1-\delta^2}<1,
\]
with
\[
\delta=\liminf_{x\to 0}\frac{f(x)}{\lV x\rV\,\lV\nabla f(x)\rV},
\]
so that $\{x^k\}_{k\in\na}$ converges linearly, with asymptotic constant bounded by $\sqrt{1-\delta^2}$.
\end{corollary}

\begin{proof}
Convergence of $\{x^k\}_{k\in\na}$ follows from \Cref{t1}. Since $x^{k+1}=C^S(x^k)$, we invoke \Cref{p11},
observing that $\liminf_{x\to 0}\theta(x)=\delta$, and taking square root and $\limsup$ in \eqref{u5}:
\[
\limsup_{k\to\infty}\frac{\lV x^{k+1}-x^*\rV}{\lV x^k-x^*\rV}\le\sqrt{1-\liminf_{k\to\infty}\theta(x^k)^2}=\sqrt{1-\delta^2}<1,
\]
using \eqref{u4} and  Assumption B3.
\end{proof}

In \cite{Arefidamghani:2021} it was shown that Assumption B3 holds in several cases, \emph{e.g.}, when $f$ is of class ${\cal C}^2$ and
the Hessian $\nabla^2 f(0)$ is positive definite, in which case  
\[\delta\ge\frac{1}{2}\frac{\lambda_{\min}}{\lambda_{\max}},
\]
where $\lambda_{\max}, \lambda_{\min}$ are the largest and smallest eigenvalues of $\nabla^2 f(0)$,
or when $f(x)=\varphi(\lV x\rV)$, where $\varphi:\re\to\re$ is a convex function of class ${\cal C}^r$, satisfying 
$\varphi(0)=\varphi'(0)=0$, in which case $\delta\ge 1/p$, where $p\le r$ is defined as $p=\min\{j:\varphi^{(j)}\ne 0\}$.

In all these instances, in view of \Cref{p10,c1}, 
the CARM sequence converges linearly, while the MAAP one converges sublinearly. If we look at the formulae for $T^S$ and $C^S$,
in \eqref{u1} and \eqref{uu8}, we note that both operators move from $(x,0)$ in the direction $(\nabla f(x),0)$ but with
different step-sizes. Looking now at \eqref{u2} and \eqref{u4}, we see that the relevant factors of these step-sizes, for
$x$ near $0$, are $f(x)/\lV x\rV$ and $f(x)/(\lV x\rV\,\lV\nabla f(x)\rV)$. Since we assume that $\nabla f(0)=0$, the first one
vanishes near 0, inducing the sublinear behavior of MAAP, while the second one, in rather generic situations, will stay away
from $0$. It is the additional presence of $\lV\nabla f(x)\rV$ in the denominator of $\theta(x)$ which makes all the difference.

Now we analyze the second family, which is similar to the first one, excepting that condition B1 is replaced by the following one:

\begin{listi}
  \item[B1')] $f(0)<0$.
 \end{listi}

We also make a further simplifying assumption, which is not essential for the result, but keeps the calculations simpler.
We take $f$ of the form $f(x)=\varphi(\lV x\rV)$ with $\varphi:\re\to\re$. Rewriting B1', B2 in terms of $\varphi$, we
assume that 
\begin{listi}
\item $\varphi:\re\to\re$ is strictly convex and of class ${\cal C} ^1$, 
\item $\varphi(0)<0$, 
\item $\varphi'(0)=0$.
\end{listi}
 
This form of
$f$ gives a one-dimensional flavor to this family. Now, $0\in\re^{n+1}$ cannot be the limit point of the MAAP or the CARM sequences:
$0$ is still the unique minimizer of $f$, but since $f(0)<0$, $0\notin \partial K$, while the limit points of the sequences,
unless they are finite (in which case convergence rates make no sense), do belong to the boundary of $K$. Hence, both $f$ and
$\nabla f$ do not vanish at such limit points, implying that both $\varphi$ and $\varphi'$ are nonzero at the norms of the
limit points. We have the following result for this family.
 
 \begin{proposition}\label{p12} 
Assume that  $U,K\subset\re^{n+1}$ are defined as $U=\{(x,0): x\in\re^n\}$ and 
$K= {\rm epi}(f)$ where $f(x)=\phi(\lV x\rV)$ and $\phi$ satisfies \emph{(i)--(iii)}. Let $C^S,T^S$ be as defined in \eqref{aa7},
and $(x^*,0), (z^*,0)$ the limits of the sequences $\{x^k\}_{k\in\na}$, $\{z^k\}_{k\in\na}$ generated
by CARM and MAAP, starting from some $(x^0,0)\in U$, and some $(z^0,w)\in\re^{n+1}$, respectively. Then
\begin{equation}\label{uu11}  
\lim_{x\to z^*}\frac{\lV T(x,0)-(z^*,0)\rV}{\lV (x,0)-(z^*,0)\rV}=
\frac{1}{1+\phi'(\lV z^*\rV)^2}
\end{equation}
and
\begin{equation}\label{uu12}
\lim_{x\to z^*}\frac{\lV C(x,0)-(x^*,0)\rV}{\lV (x,0)-(x^*,0)\rV}=0.
\end{equation}
\end{proposition}

\begin{proof} We start by rewriting the formulae for $C^S(x), T^S(x)$ in terms of $\varphi$. 
We also define $t:=\lV x\rV$. Using \eqref{uu1}, \eqref{u1}, \eqref{u4}
and \eqref{u5}, we obtain 
\begin{equation}\label{u11}
T^S(x,0)=\left(\left[1-\frac{\varphi(\lV x\rV)\varphi'(\lV x\rV)}{(\varphi'(\lV x\rV)^2+1)\lV x\rV}\right]x,0\right)=
\left(\left[1-\frac{\varphi(t)(\varphi'(t)}{(\varphi'(t)^2+1)t}\right]x,0\right)
\end{equation}
and
 \begin{equation}\label{u12}
C^S(x,0)=
\left(\left[1-\frac{\varphi(\lV x\rV)}{\varphi'(\lV x\rV)\lV x\rV}\right]x,0\right)=\left(\left[1-\frac{\varphi(t)}{\varphi'(t)t}\right]x,0\right).
\end{equation}
Note that $x,T^S(x), C^S(x)$ are collinear (the one-dimensional flavor!), so that the same happens with $x^*,z^*$. 
Let $r:=\lV x^*\rV, s:=\lV z^*\rV$, so that $x^*=(r/t)x, z^*=(s/t)x$. Then, using \eqref{u11}, \eqref{u12}, we get
\begin{align}
\frac{\lV T^S(x,0)-(z^*,0)\rV}{\lV (x,0)-(z^*,0)\rV}&=\frac{t-r-\frac{\varphi(t)\varphi'(t)}{\varphi'(t)^2+1}}{t-r}=
\left[1-\frac{\varphi(t)}{t-r}\right]\left[\frac{\varphi'(t)}{\varphi'(t)^2+1}\right]\\&=
\left[1-\frac{\varphi(t)-\varphi(r)}{t-r}\right]\left[\frac{\varphi'(t)}{\varphi'(t)^2+1}\right],\label{u13} 
\end{align}
and
\begin{align}
\frac{\lV C^S(x,0)-(x^*,0)\rV}{\lV (x,0)-(x^*,0)\rV}&=\frac{t-s-\frac{\varphi(t)}{\varphi'(t)}}{t-s}=
1-\left[\frac{\varphi(t)}{t-s}\right]\frac{1}{\varphi'(t)}\\&=
1-\left[\frac{\varphi(t)-\varphi(r)}{t-s}\right]\frac{1}{\varphi'(t)},\label{u14}
\end{align}
using in the last equalities of \eqref{u13} and \eqref{u14} the fact that $\varphi(r)=\varphi(s)=0$, which results from $f(x^*)=f(z^*)=0$.
Now we take limits with $x\to z^*$, $x\to x^*$ in the leftmost expressions of \eqref{u13}, \eqref{u14}, 
which demands limits with $t\to s$, $t\to r$
in the rightmost expressions of them. 
\begin{align}
\lim_{x\to x^*}\frac{\lV T^S(x,0)-(z^*,0)\rV}{\lV (x,0)-(z^*,0)\rV}&=
\lim_{t\to r}\left[1-\frac{\varphi(t)-\varphi(r)}{t-r}\right]\left[\frac{\varphi'(t)}{\varphi'(t)^2+1}\right]\\&=
1-\frac{\varphi'(r)^2}{\varphi'{r}^2+1}=\frac{1}{\varphi'(r)^2+1}=\frac{1}{\varphi'(\lV z^*\rV)^2+1},
\end{align}
and 
\[
\lim_{x\to z^*}\frac{\lV C^S(x,0)-(x^*,0)\rV}{\lV (x,0)-(x^*,0)\rV}=
\lim_{t\to s}\left[1-\frac{\varphi(t)-\varphi(r)}{t-s}\right]\frac{1}{\varphi'(t)}=
1-\frac{\varphi'(s)}{\varphi'(s)}=0,
\]
completing the proof.
\end{proof}

\begin{corollary}\label{c2}
Under the assumptions of \Cref{p12}, the sequence generated by MAAP converges Q-linearly to a point $(x^*,0)\in K\cap U$, 
with asymptotic constant equal to $1/(1+\varphi'(\lV x^*\rV)^2)$, and the sequence generated by CARM converges superlinearly.
\end{corollary}
\begin{proof} 
Recall that if $\{(x^k,0)\}_{k\in\na}$ is the MAAP sequence, then $(x^{k+1},0)=T^S(x^k,0)$, and if $\{(z^k,0)\}_{k\in\na}$ is the CARM 
sequence, then $(z^{k+1},0)=C^S(z^k,0)$. Recall also that for both sequences the last components of the iterates vanish because
$\{x^k\}_{k\in\na}, \{z^k\}_{k\in\na}\subset U$. Then the result follows immediately from \eqref{uu11} and \eqref{uu12} in \Cref{p12}.
\end{proof} 
    
We mention that the results of \Cref{c2} coincide with those obtained in Corollary 4.11 of \cite{Arefidamghani:2021} for the
sequences generated by MAP and CRM applied to the same families of instances of CFP, showing that the convergence rate
results of the exact methods are preserved without any deterioration also in these cases.

\section{Numerical comparisons}\label{sec:numerical}

In this section, we perform numerical comparisons between CARM, MAAP, CRM and MAP. These methods are employed for solving the particular CFP  of finding a common point  in the intersection of finitely many ellipsoids, that is, finding 
\[\label{eq:ellipsoids} \bar x \in \mathcal{E} = \bigcap_{i=1}^{m}\mathcal{E}_i\subset \re^n,\] with each ellipsoid  $\mathcal{E}_i$ being given by
\[\mathcal{E}_{i}:=\left\{x\in \re^n : g_i(x) \leq 0 \right\}, \text{ for } i=1,\ldots,m, \]
where $g_i:\re^n\to\re$ is given by $g_i(x) =  x^{\top} A_{i} x +2 x^{\top} b^{i} -  \alpha_{i}$, each $A_{i}$ is a symmetric positive definite matrix, $b^{i}$ is an $n$-vector, $\alpha_{i}$ is a positive scalar. 

Problem \eqref{eq:ellipsoids}  has importance own its own (see \cite{Lin:2004,Jia:2017}) and both CRM and MAP are suitable for solving it. Nevertheless, the main motivation for tackling it with approximate projection methods is that the computation of exact projections onto ellipsoids is a formidable burden for any algorithm to bear. Since  the gradient of each $g_i$ is easily available, we can consider the separable operators given in \Cref{ex1,ex2} and use CARM and MAAP to solve problem \eqref{eq:ellipsoids} as well. What is more, the experiments illustrate that, in this case, CARM handily outperforms CRM, in terms of CPU time, while still being competitive in terms of iteration count. 
The exact projection onto each ellipsoid is so demanding that even MAAP has a better CPU time result than CRM.

The four methods are employed upon Pierra's product space reformulation, that is, we seek a point $\mathbf{x}^* \in \mathbf{K}\cap \mathbf{D}$,  where $\mathbf{K} \coloneqq \mathcal{E}_{1}\times \mathcal{E}_{2}\times \cdots \times \mathcal{E}_{m}$ and $\mathbf{D}$ is the diagonal space. For each sequence $\{\mathbf{x}^k \}_{k\in\na}$ that we generate, we consider the tolerance $\varepsilon \coloneqq \num{e-6}$ and use as stopping criteria the \emph{gap distances}
\[
\lVert \mathbf{x}^k - P_{\mathbf{K}} (\mathbf{x}^k)\rVert <  \varepsilon \quad \text{ or }  \quad  \lVert \mathbf{x}^k - P^S_{\mathbf{K}} (\mathbf{x}^k)\rVert <  \varepsilon,
\]
where $P_{\mathbf{K}} (\mathbf{x}^k)$ is utilized for CRM and MAP, and $P^S_{\mathbf{K}} (\mathbf{x}^k)$ is used for CARM and MAAP. Note that if the correspondent criterion is not met in a given iteration, the projection computed is employed to yield the next step.    We also set  the maximum number of iterations as \num{50000}.  

To execute our tests, we randomly generate \num{160} instances of \eqref{eq:ellipsoids} in the following manner. We range the dimension size $n$ in $\{10, 50, 100, 200\}$ and for each $n$ we took  the number $m$ of underlying sets varying in  $\{5, 10, 20 ,50\}$. For each of these \num{16} pairs $(m,n)$ we build \num{10} randomly generated instances of \eqref{eq:ellipsoids}. Each matrix $A_i$ is of the form $A_i = \gamma \Id +B_i^\top B_i$, with  $B_i \in\re^{n\times n}$, $\gamma \in \re_{++}$.   Matrix $B_i$ is a sparse matrix sampled from the standard normal distribution  with sparsity density  $p=2 n^{-1}$  and each vector $b^i$ is sampled from the uniform distribution between $[0,1]$. We then choose each $\alpha_i$ so that $\alpha_i > (b^i)^\top Ab^i$, which ensures that $0$ belongs to every $\mathcal{E}_i$, and thus \eqref{eq:ellipsoids} is feasible. The initial point $x^0$ is of the form $(\eta,\eta,\ldots, \eta)\in\re^n$, with $\eta$ being  negative and $\lv \eta\rv$ sufficient large, guaranteeing  that  $x^0$ is far from all $\mathcal{E}_i$’s.

The computational experiments were performed on an Intel Xeon W-2133 3.60GHz with 32 GB of RAM running Ubuntu 20.04 and using \texttt{Julia v1.5}  programming language~\cite{Bezanson:2017}.  The codes for our experiments are fully available in \url{https://github.com/lrsantos11/CRM-CFP}.

We remark that, as CRM and MAP rely on the computation of exact projections, \texttt{ALGENCAN}~\cite{Birgin:2014}, an Augmented Lagrangian algorithm implemented in \texttt{Fortran} (wrapped in \texttt{Julia} using \texttt{NLPModels.jl}~\cite{Siqueira:2019}) was used in our code to compute  projections onto the ellipsoids. Each projection was found by solving the correspondent quadratic minimization problem with quadratic constraints (the $g_i$‘s).

\begin{figure}[ht!]
  \centering
  \includegraphics[scale=0.6]{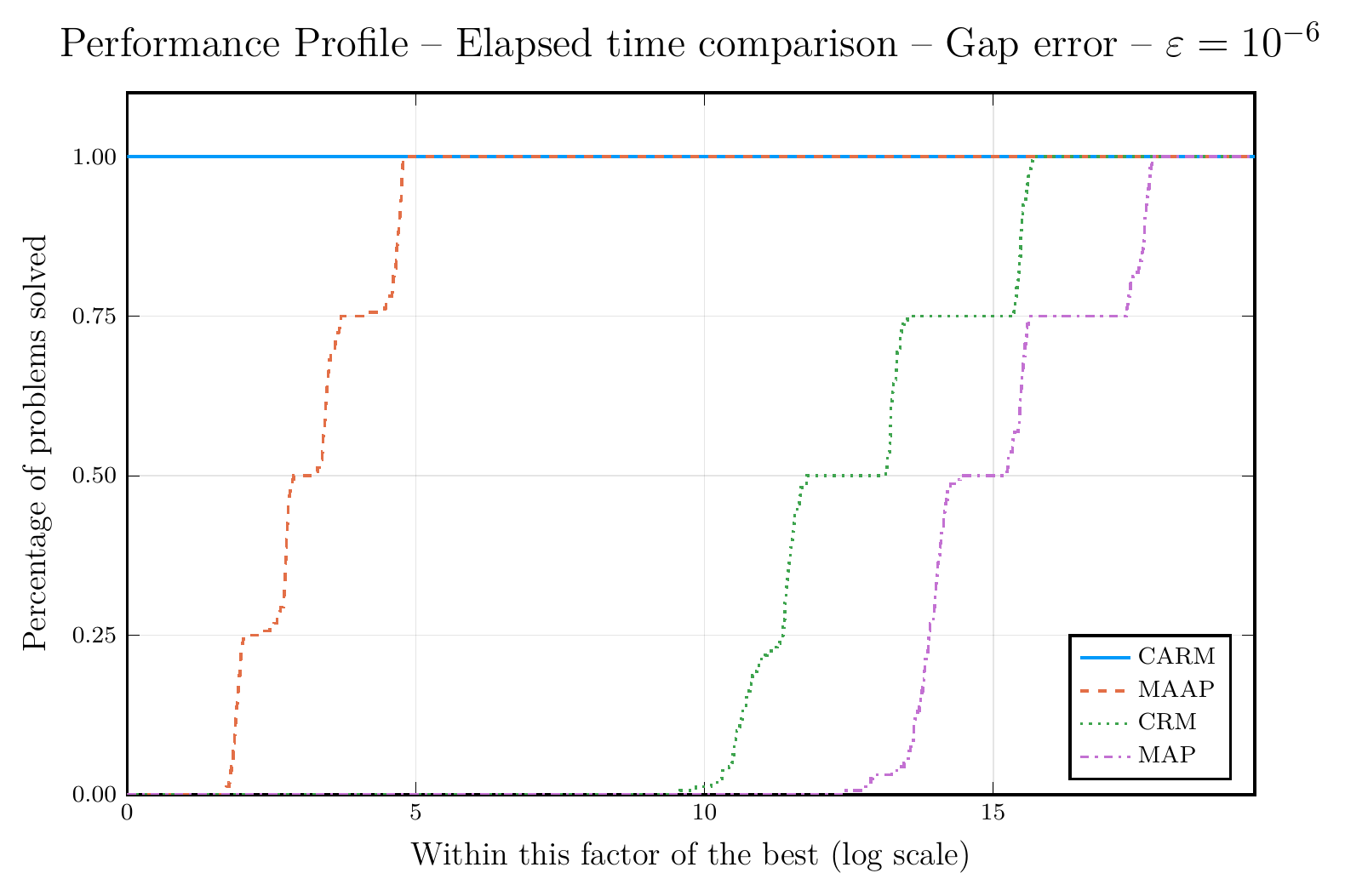}
  \caption{Performance profile of experiments with ellipsoidal feasibility  -- CARM, MAAP, CRM and MAP}
  \label{fig:pp-Ellipsoid}
\end{figure}

The results are summarized in~\Cref{fig:pp-Ellipsoid}  using a performance profile~\cite{Dolan:2002}.  Performance profiles allow one to compare different methods on problems set with respect to a performance measure.  The vertical axis indicates the percentage of problems solved, while the horizontal axis indicates, in log-scale, the corresponding factor of the performance index  used by the best solver. In this case, when looking at CPU time (in seconds), the performance profile shows that CARM always did better than the other three methods.  The picture also shows that MAAP took less time than CRM and MAP.  We conclude this examination by presenting, in \Cref{table:ellipsoid}, the following descriptive statistics of the benchmark of CARM, MAAP, CRM and MAP: mean, maximum (max), minimum (min) and standard deviation (std) for iteration count (it) and CPU time in seconds (CPU (s)). In particular, CARM was, in average,  almost $\num{3000}$ times faster than CRM.

\begin{table}[ht!]
  \centering
  \caption{Statistics of the experiments  (in number of iterations and CPU time)}
  \label{table:ellipsoid}
  \sisetup{
table-parse-only,
table-figures-decimal = 4,
table-format = +3.4e-1,
table-auto-round,
%
}
\begin{tabular}{lrSSSS}
\toprule  \textbf{Method} & &   {\textbf{mean}} &  {\textbf{max}} &   {\textbf{min}} &   {\textbf{std}} \\
\cmidrule(lr){3-6}
CARM & \texttt{it} & 6.4875   & 8           & 6            & 0.5259    \\
 & \texttt{CPU (s)} & \num{1.4632e-3} & \num{8.4866e-3}  & \num{9.3599e-5}   &\num{ 1.8643}  \\
\cmidrule(lr){2-6}
MAAP &  \texttt{it} & 260.75   & 689         & 54           & 213.346   \\
 & \texttt{CPU (s)} & \num{2.7132e-2} &  0.9736 & \num{3.2268e-4}   & \num{4.7248e-2}   \\
\cmidrule(lr){2-6}
CRM & \texttt{it}   & 4.35     & 6           & 3            & 0.8256    \\
 & \texttt{CPU (s)} & 39.592   & 358.846   & 0.0872    & 82.9985   \\
\cmidrule(lr){2-6}
MAP   & \texttt{it} & 257.856  & 671         & 54           & 211.537   \\
 & \texttt{CPU (s)} & 182.567  & 1616.24   & 1.0315     & 391.669     \\
\bottomrule
\end{tabular}
  
\end{table}

\section{Concluding remarks}

In this paper, we have introduced a new circumcenter iteration for solving convex feasibility problems. The new method is called CARM, and it utilizes outer-approximate projections instead of the exact ones taken in the original CRM.

We have drawn our attention to questions on whether similar convergence results known for CRM could be generalized for CARM. We have derived many positive theoretical statements in this regard. For instance, the convergence of CARM was proven, and linear rates were achieved under error bound conditions. In addition to that, we presented numerical experiments in which subgradient approximate projections were employed. This choice of approximate projections is a particular case of the ones covered by CARM. The numerical results show CARM to be much faster in CPU time than the other methods we compared it with, namely, the pure CRM, the classical  MAP, and an approximate version of MAP called MAAP.


\begin{backmatter}

\section*{Acknowledgements}

The authors would like to thank the two anonymous referees for their valuable suggestions which 
improved this manuscript.

\section*{Funding}

RB was partially supported by the \emph{Brazilian Agency Conselho Nacional de Desenvolvimento Cient\'ifico e Tecnol\'ogico} (CNPq), Grants 304392/2018-9 and 429915/2018-7;  YBC was partially supported by the \emph{National Science Foundation} (NSF), Grant DMS -- 1816449.


\section*{Availability of data and materials}

Not applicable.


\section*{Competing interests}
The authors declare that they have no competing interests.


\section*{Authors' contributions}
All authors read and approved the final manuscript.



\bibliographystyle{bmc-mathphys} 
\bibliography{refs}      









\end{backmatter}
\end{document}